\documentclass[11pt,a4paper,reqno]{amsart}
\usepackage{amsfonts}
\usepackage{amsthm}
\usepackage{amsmath}
\usepackage{amscd}
\usepackage[latin2]{inputenc}
\usepackage{t1enc}
\usepackage[mathscr]{eucal}
\usepackage{indentfirst}
\usepackage{graphicx}
\usepackage{graphics}
\usepackage{pict2e}
\usepackage{dsfont}
\usepackage{epic}
\numberwithin{equation}{section}
\usepackage[margin=2.9cm]{geometry}
\usepackage{epstopdf} 
\usepackage{amssymb}
\usepackage{setspace}
\usepackage{enumerate}
\usepackage{bigstrut}
\usepackage{multirow}
\usepackage{mathtools,xparse}
\usepackage{mathrsfs}
\usepackage{hyperref}
\newtheorem{theorem}{Theorem}[section]
\newtheorem{corollary}[theorem]{Corollary}
\newtheorem{lemma}[theorem]{Lemma}
\newtheorem{proposition}[theorem]{Proposition}
\theoremstyle{definition}
\newtheorem{definition}[theorem]{Definition}
\newtheorem{remark}[theorem]{Remark}
\allowdisplaybreaks
\usepackage{chngcntr}

\def\be{\begin{equation}}
\def\ee{\end{equation}}
\def\bepro{\begin{proposition}}
\def\enpro{\end{proposition}}
\def\belemma{\begin{lemma}}
\def\enlemma{\end{lemma}}
\def\it{\textit}

\newcommand{\E}{\mathop{\mathbb{E}}}
\newcommand{\1}{\mathds{1}}

\newcommand{\R}{\mathbb{R}}
\newcommand{\norm}[1]{\left\lVert#1\right\rVert}

\begin{document}

\title{Hypoelliptic diffusions with singular drifts}

\author{Kyeongsik Nam}

\begin{abstract}
We establish the well-posedness  of  stochastic differential equations possessing  degenerate diffusions and singular drifts. We prove that SDEs defined on the homogeneous Carnot group, whose hypoelliptic diffusion part is given by the \it{horizontal Brownian motion}, admit a unique strong solution for a large class of singular  drifts. It considerably generalizes the classical well-posedness results of  singular SDEs with non-degenerate diffusions. It also provides an intermediate result between the Cauchy-Lipschitz theorem in ordinary differential equations and the result proved by Krylov and R\"ockner \cite{KR}, which states the well-posedness of SDEs with the additive noise and singular drifts. 
\end{abstract}

\address{ Department of Mathematics, Evans Hall, University of California, Berkeley, CA
94720, USA} 

\email{ksnam@math.berkeley.edu}


 
\maketitle

\section{Introduction}
The theory of stochastic differential equations (SDE)
\be \label{SDE1}
\begin{cases}
dX_t=b(t,X_t)dt+\sigma(t,X_t)dB_t, \quad 0\leq t\leq T, \\
X_0=x_0\in \R^d,
\end{cases}
\ee
where $B_t$ denotes a standard Brownian motion, has been central in the probability  theory for a long time due to its wide applications in partial differential equations (PDE). In particular,  there are several interesting connections between SDEs and PDEs in the fluid dynamics. For instance, the 3D incompressible Navier-Stokes equations  can be represented in terms of the  stochastic Lagrangian systems (see \cite{nse}). Also, Rezakhanlou  \cite{re, re2}  established the probabilistic interpretations of a certain class of solutions to the Navier-Stokes equations using the Hamiltonian dynamics approach. Therefore, it is crucial to establish the qualitative properties of solutions to SDEs \eqref{SDE1} for a broad class of coefficients $b$ and $\sigma$ due to its broad applications to other problems.

In particular, the theory of   SDEs \eqref{SDE1} with singular drifts $b$ and the additive noise ($\sigma=\text{Id}$) has been successfully applied to answer the well-posedness questions of PDEs having singular coefficients. For example, Flandoli et al. \cite{gub} showed that for singular vector fields $b$ on $\R^d$ satisfying $b\in L^\infty([0,T],C^\alpha_x)$ and $\text{div}b\in L^p([0,T]\times \R^d)$ with $\alpha\in (0,1), p\in (2,\infty)$, the following stochastically perturbed transport equation: 
\begin{align} \label{transport}
d_tu(t,x)+b(t,x)\cdot Du(t,x)dt + \sum_{i=1}^d e_i\cdot Du(t,x)\circ dB_t^i = 0,\quad u(0,\cdot)=u_0\in L^\infty
\end{align}
($e_i$'s are standard vectors in the Euclidean space and $\circ$ denotes the Stratonovich integral) admits a unique $L^\infty$-weak solution. This result can be interpreted as an regularization effect by the noise since in the absence of randomness, the classical transport equation may have several weak solutions under the same condition on $b$. Authors proceed with the proof by first establishing the fact that SDE \eqref{SDE1} with the additive noise and a  H\"older continuous drift $b\in L^\infty([0,T],C^\alpha)$ admits a unique strong solution possessing a rich regularity. The improved regularity of a solution to SDE \eqref{SDE1} plays a crucial role to obtain the commutator estimate which is a key ingredient to establish the uniqueness of a weak solution to the equation \eqref{transport}.

 In the absence of noise ($\sigma=0$), it is a classical fact in ordinary differential equations (ODE) that Lipschitz continuity of $b(t,\cdot)$ and continuity of $b(\cdot,x)$ ensure the existence and uniqueness of a solution to \eqref{SDE1} (see also \cite{amb, lions} for the theory of Lagrangian flows). On the other hand, in the presence of additive noise ($\sigma = Id$), the well-posedness of SDE \eqref{SDE1} is known for a broader class of drifts $b$. For instance, Veretennikov \cite{ver} proved the well-posedness  for the bounded drifts $b$ in dimension one. Breakthrough was made by Krylov and R\"ockner \cite{KR}: SDE \eqref{SDE1} with the additive noise is well-posed for singular drifts $b$ satisfying
\be \label{KR}
b\in L^q([0,T],L^p(\R^d)), \quad \frac{2}{q}+\frac{d}{p}<1, \ 1<p,q<\infty.
\ee
This is a striking result considering that the well-posedness of SDE can be obtained even when no regularity condition is imposed on the singular and unbounded drift $b$.
The \it{Yamada-Watanabe Principle} \cite{YW,YW1} plays a crucial role in their argument: the existence of a weak solution combined with the uniqueness of a strong solution to \eqref{SDE1} implies the existence of a strong solution and uniqueness of a weak solution to \eqref{SDE1} (see Definition \ref{definition2} and Theorem \ref{a.3} for details). The result \cite{KR} has been successfully extended to the  general type of the multiplicative non-degenerate noise \cite{zhang0, zhang1}. Also, qualitative properties of a solution to singular SDE \eqref{SDE1} with the non-degenerate diffusion have been obtained: for instance, the Sobolev regularity is proved under the assumption \eqref{KR} (see for example \cite{ff2}). We refer to \cite{att,gub2,da, zhang2, malliavin,zhang3} for the further results in this direction.

A natural and important question is whether or not one can establish the well-posedness property of SDE \eqref{SDE1} for a large class of  singular drifts $b$ even when the diffusion coefficient is degenerate. To gain some advantages from the randomness, we require the noise to be hypoelliptic in order that the  regularization effect happens. Consider the following Stratonovich SDE with  smooth coefficients possessing the hypoelliptic diffusion:
\be \label{SDE0}
\begin{cases}
dX_t=b(t,X_t)dt+\sum_{i=1}^m Z_i(X_t) \circ dB^i_t, \\
X_0=x_0.
\end{cases}
\ee 
 Here, $B^i$'s are independent standard one dimensional Brownian motions. The reason why SDEs of the Stratonovich formulation rather than  It\^o formulation \eqref{SDE1} is more suitable to describe the hypoelliptic diffusion is that it has a nice control theoretical interpretation thanks to the Stroock-Varadhan support theorem (see \cite{support}): if we denote $x^h$ by a solution of the following ODE:
\begin{align*}
dx^h_t=b(t,x^h_t)dt+\sum_{i=1}^m Z_i(x^h_t)dh^i,
\end{align*}
then the support of the law of a solution $X_t$ of SDE \eqref{SDE0} is  the closure of a set $\{x^h\,|\,\frac{dh}{dt}\in L^2([0,T],\R^d)\}$ in $C^\alpha$ topology. Here, the support of the law of $X_t$ is regarded as a measure on the path space $C^\alpha$. It is a classical theory that  under the celebrated H\"ormander's condition \cite{hor}, the regularization effect happens in the sense that the law of a solution to SDE \eqref{SDE0} possesses a smooth density. We refer to \cite{hor} for a purely analytical approach and \cite{m} for a probabilistic approach which is known as Malliavin Calculus. There are analogous results when the standard Brownian motion  is replaced with the fractional Brownian motion \cite{fbm1,fbm2}, or a rough path \cite{rough2,rough1}.

 It is  crucial to understand the qualitative properties of SDE \eqref{SDE0} with the hypoelliptic diffusion since it appears naturally in various areas of mathematics such as sub-Riemannian geometry as well as phase space problems. For instance, in the context of the sub-Riemannian geometry, several properties such as Log-Sobolev inequalities and the heat kernel estimates for the hypoelliptic diffusions have been established in \cite{bbbc,bb,bgm,b,dem}. In particular, from the functional inequality associated with the hypoelliptic diffusion, which is called  the \it{generalized curvature-dimension inequalities}, the classical PDE results such as Harnack's inequality and the Li-Yau type inequalities for the hypoelliptic operators can be obtained (see \cite{cd1} for details).
 This  shows that a qualitative theory for the  hypoelliptic diffusions can provide an essential ingredient to study hypoelliptic operators. 
 
  Therefore, it is natural and important to develop a qualitative theory for a broad class of hypoelliptic diffusions on the sub-Riemannian manifold, in particular hypoelliptic diffusions with singular drifts.  The first step to accomplish this is to establish the well-posedness result of SDE \eqref{SDE0} for a large class of singular drifts.  Note that as explained right after in the SPDE \eqref{transport}, the well-posedness result of the hypoelliptic SDE \eqref{SDE0} with singular drifts can also provide a key tool to prove the well-posedness of SPDE with singular coefficients and hypoelliptic diffusions.

Recently, several authors obtained the well-posedness results of SDEs with non-smooth drifts and the degenerate noise. For instance, Zhang \cite{zhang4} considered the phase space problem:  the stochastic Hamiltonian system
\begin{align} \label{phase}
dY_t=V_tdt,  \quad dV_t=F(t,Y_t)dt+\sigma(t,Y_t)dB_t.
\end{align} 
 It is proved in \cite{zhang4} that if a non-smooth drift $F$ possesses some Sobolev regularity, then SDE \eqref{phase} is well-posed. In addition, in \cite{hw}, authors found the conditions on the singular drifts for which the corresponding SDEs with  hypoelliptic diffusions defined on the generalized Heisenberg groups are well-posed. We also refer to  \cite{xz} for other well-posedness results  of SDEs with the non-uniformly  elliptic noise.

 However, only specific classes of SDEs are considered in all of the previously mentioned works. For example, in the case of phase space problem \eqref{phase} studied in \cite{zhang4}, the drift of a process $X_t=(Y_t,V_t)$  is of a very particular form. Also, in \cite{hw}, the underlying space on which the singular SDEs are considered is only a special form of the general Lie groups. Therefore, these results are too special to be regarded as a general theory on the well-posedness of singular SDEs with the  degenerate diffusions.

 In this paper, we provide a large class of singular drifts $b$ for which we can ensure the existence and uniqueness of a strong solution to the SDE \eqref{SDE0} with  hypoelliptic diffusions. We consider the SDEs with singular drifts  on the \it{homogeneous Carnot group} (see Section \ref{section 3} for the definition), which is an example of the sub-Riemannian manifolds. We assume that vector fields $Z_1,\cdots,Z_m$ in SDE \eqref{SDE0} are left-invariant and form a basis of the first layer of stratified Lie algebra. In the terminology of sub-Riemannian geometry, the diffusion part of SDE \eqref{SDE0} is called the \it{horizontal Brownian motion}, and it is hypoelliptic.  We prove that if a singular drift $b$ belongs to the suitable mixed-norm parabolic Sobolev spaces associated with the vector fields generating  randomness, then SDE \eqref{SDE0} is locally well-posed (see Theorem \ref{main}). Briefly speaking, we impose some regularity and integrability conditions on the drift. The conditions are expressed in terms of  \it{nilpotency} and \it{homogeneous dimension} of the group, respectively (see the condition \eqref{condition2}  for details).  
Surprisingly,  this result completely recovers the classical well-posedness result of  SDEs with the additive noise and singular drifts \cite{KR}, and considerably generalizes the previously known results (see Section \ref{section 2} for details). 
 
 To the author's knowledge, the main result of this paper provides the first well-posedness result for singular SDEs with the degenerate diffusion in a great generality. Therefore, the main result of this paper can be regarded as a starting point to study the qualitative properties of general SDEs with hypoelliptic diffusions and singular drifts. For instance, many questions about the smooth hypoelliptic diffusions that arise in the sub-Riemannian geometry mentioned before can also be recast in the context of the hypoelliptic diffusions with singular drifts.

We should point out that the main result in this paper provides a  beautiful intermediate result between the well-posedness result in the classical ODE theory (Cauchy-Lipschitz theorem) and the well-posedness result of SDEs with the additive noise \cite{KR} mentioned above. Roughly speaking,  as the noise gets degenerate in the sense that a nilpotency of the homogeneous Carnot group converges to infinity, the conditions imposed on the singular drift formally become the Lipschitzness condition. On the other hand, if the homogeneous Carnot group is just a Euclidean space $\R^d$  and vector fields generating the randomness are the standard vector fields on $\R^d$ (which corresponds to the additive noise case), then Theorem \ref{main} covers the classical result of Krylov and R\"ockner \cite{KR}. More details about these interpretations can be found in Section \ref{section 2}. 

The structure of the paper is as follows. We state the main theorem and provide a brief outline of the proof in Section \ref{section 2}. In Section \ref{section 3}, we  briefly overview the theory of analysis on the nilpotent Lie group. Also, as applications of the classical Calder\'on-Zygmund theory and the theory of subelliptic operators \cite{f2,f3}, we establish the well-posedness of the parabolic equation involving the sub-Laplacian. In Section \ref{section 4}, we study the Kolmogorov equation  associated with the SDE \eqref{SDE0}, which is a key ingredient to prove the uniqueness of a strong solution to SDE \eqref{SDE0}. In Section \ref{section 5}, we establish the Krylov-type estimates and then prove the It\^o's formula for  singular functions. Finally, applying the Zvonkin's transformation method \cite{zvo}, we conclude  the proof of the main theorem  in Section \ref{section 6}.

Throughout this paper, $\nabla=(\frac{\partial}{\partial x_1},\cdots, \frac{\partial}{\partial x_N})$ denotes the standard gradient on the Euclidean space $\R^N$. $B_t$ denotes the standard Brownian motion on a filtered space $(\Omega, \mathcal{F}, \mathcal{F}_t, P)$ with the filtration $\mathcal{F}_t=\sigma\{B_r|0\leq r\leq t\}$. Also, we say that for functions $f$ and $g$ defined on $\R \times \mathbb{G}$, $f$ is a version of $g$ provided that $f=g$ for $(t,x)$-a.e. Furthermore, $f\in C_b$ means that $f$ is a bounded and continuous function. Finally, for a matrix $A$, $\norm{A}$ denotes the Hilbert-Schmidt norm. 

\section{Main results} \label{section 2}
Consider the homogeneous Carnot group $\mathbb{G}=(\R^N,\circ,D(\lambda))$, where $\{D(\lambda)\}_{\lambda>0}$ denotes the dilation structure. Assume that $\mathbb{G}$ has a homogeneous dimension $Q$, nilpotency $r$, and $Z_i$'s ($1\leq i\leq m$) are left invariant vector fields that form a basis of the first layer of the Lie algebra $\mathfrak{g}$ (see Section \ref{section 3} for details). 

 Let us consider the following Stratonovich SDE on the homogeneous Carnot group $\mathbb{G}$:
\be \label{SDE}
\begin{cases}
dX_t=b(t,X_t)dt+\sum_{i=1}^m Z_i(X_t) \circ dB^i_t, \\
X_0=x_0.
\end{cases}
\ee 
Here, $\mathbb{G}$ is identified with $\R^N$ and $Z_i$'s ($1\leq i\leq m$) are regarded as vector fields on $\R^N$.
Also, suppose that two exponents $p$ and $q$ satisfying
\be \label{pq}
\frac{2}{q}+\frac{Q}{p}<1, \quad 1<p,q<\infty,
\ee
are given, and that the drift $b$ satisfies
\be \label{condition1}
b\in \text{span}\{Z_1,Z_2,\cdots Z_m\}, \\
\ee
\be \label{condition2}
Z_I b^i \in L^q([0,T],L^p(\mathbb{G}))\quad \text{for} \quad 1\leq i\leq m,\ |I|\leq r-1.\\
\ee
Here, $Z_If$ denotes $Z_{i_1}\cdots Z_{i_k}f$ (distribution derivatives) for a multi-index $I=(i_1,\cdots,i_k)$, $1\leq i_1,\cdots,i_k\leq m$, and $b^i$'s ($1\leq i\leq m$) are given by the expression:
\begin{align*}
b=\sum_{i=1}^m b^iZ_i.
\end{align*} Our main result Theorem \ref{main} below claims that one can  construct a (unique) solution to SDE \eqref{SDE} for a broad class of singular drifts $b$.

\begin{theorem} \label{main}
Let $(\mathbb{G},\circ,D(\lambda))$ and $\{Z_i| 1\leq i\leq m\}$ be as above. Assume that a singular drift  $b$ satisfies the conditions \eqref{condition1} and \eqref{condition2} for the exponents $p,q$ satisfying \eqref{pq}. Then, for some open set $U$ containing $x_0$, there exists a unique strong solution $X_t$ to SDE \eqref{SDE} before the time at which $X_t$ exits $U$.
\end{theorem}

Note that since Theorem \ref{main} is a local statement, the condition \eqref{condition2} can be replaced by the weaker local $L^p$ spaces:
\begin{align*}
Z_I b^i \in L^q([0,T],L^p_{\text{loc}}(\mathbb{G}))\quad \text{for} \quad 1\leq i\leq m,\ |I|\leq r-1.
\end{align*}

Remarkably, Theorem \ref{main} provides a  beautiful intermediate well-posedness result between  ODE case (absence of the randomness) and SDE with the additive noise case \cite{KR}.  Recall that as explained in the introduction,   SDE with an additive noise ($\sigma=Id$):
\begin{align} \label{additive}
dX_t=b(t,X_t)dt+dB_t,\quad X_0=x_0,
\end{align}
admits a unique strong solution if a singular drift $b$ satisfies the condition \eqref{KR}.
This result  can be regarded as a special case of  Theorem \ref{main}. In fact, if we consider the homogeneous Carnot group given by
\begin{align*}
(\mathbb{G},\circ)=(\R^d,+), \quad Z_i=\frac{\partial}{\partial x_i}\ (1\leq i\leq d),
\end{align*}
with a standard dilation structure, then the homogeneous dimension  $Q$ is   $d$  and nilpotency $r$ is 1. Also, SDE \eqref{SDE} becomes the additive noise SDE. Therefore, by comparing \eqref{KR} with the conditions \eqref{pq}-\eqref{condition2}, Theorem \ref{main} can be regarded as a considerable generalization of the well-posed result of singular SDEs from the non-degenerate diffusion case $r=1$ to the  degenerate diffusion cases $r>1$. 

In addition, one can formally check that in the  limit $r\rightarrow \infty$, Theorem \ref{main}  covers the classical well-posed result in the ODE theory.  Note that if we write $W^{k,p}$ for the standard Sololev spaces and $S^{k,p}$ for the Sobolev spaces with respect to vector fields $\{Z_i| 1\leq i\leq m\}$: 
\begin{align*}
S^{k,p}(\mathbb{G}):=\{f|Z_If \in L^p(\mathbb{G}), |I|\leq k\},
\end{align*}
then for $1<p<\infty$, the following relation holds:
\begin{align} \label{local embedding}
W^{k,p}_{\text{loc}} \subset S^{k,p}_{\text{loc}} \subset W^{k/r,p}_{\text{loc}}
\end{align} 
 (see \cite[Theorem 2]{f3}).
 Also, it is obvious that under the conditions \eqref{condition1} and \eqref{condition2}, each Euclidean coordinate of a singular drift $b$ belongs to the space $L^q([0,T],S^{r-1,p}_{\text{loc}})$.  
 Thus, using this fact and the relation \eqref{local embedding}, one can conclude that
\begin{align*}
b\in L^q([0,T],W^{\frac{r-1}{r},p}_{\text{loc}}).
\end{align*}
 Thus, if the noise becomes more degenerate in the sense that $r \rightarrow \infty$, it follows that $Q\rightarrow \infty$, and thus we have
\begin{align*}
\frac{r-1}{r} \rightarrow 1, \quad p \rightarrow \infty,
\end{align*}
due to the condition \eqref{pq}. Since the space $L^\infty([0,T],W^{1,\infty}_x)$ is the class of drifts for which the corresponding ODE
\begin{align*}
x'(t)=b(t,x(t)), \quad x(0)=x_0
\end{align*}
is well-posed, the formal limit $r\rightarrow \infty$ in Theorem \ref{main} covers the classical well-posedness result in the ODE theory.

\begin{remark}
In the terminology of the sub-Riemannian geometry, the condition \eqref{condition1} means that the vector field $b$ belongs to the \it{horizontal distribution} which is  the linear span of $Z_1,\cdots,Z_m$. It is completely non-integrable, and arbitrary two points can be connected by a horizontal path according to the Chow-Rashevskii theorem.  A stochastic process whose generator is $\frac{1}{2}\sum_{i=1}^m Z_i^2$ is called the \it{horizontal Brownian motion}. Theorem \ref{main} claims that if we perturb the horizontal Brownian motion by a certain class of singular drifts contained in the horizontal distribution, then it is still locally well-defined. We refer to \cite{mont} for the monograph of the sub-Riemannian geometry.
\end{remark}

\begin{remark} \label{remark1.3}
Since Theorem \ref{main} is a local statement, throughout the paper, we  assume that each $b^i$ has a compact support, which is  uniform in $t$.
Note that in the case $r=1$, we have $(\mathbb{G},\circ)=(\R^N,+)$ and $Z_i = \frac{\partial}{\partial x_i}$ for $1\leq i\leq N$, which corresponds to the additive noise case, and this case is considered  in \cite{KR}.  Therefore, throughout this paper, we only consider the case $r>1$ which corresponds to the  degenerate diffusion case. 

 The condition  \eqref{condition2} implies that $b^i(t,\cdot)\in S^{r-1,p}(\mathbb{G})$ for $t$-a.e. Since $(r-1)p \geq p > Q$ under the condition \eqref{pq} and $r>1$, according to the Sobolev embdding Theorem \ref{2.8}, there is a version of $b$ such that $b(t,\cdot)$ is continuous for $t$-a.e. In this paper, we prove Theorem \ref{main} for such drifts.
\end{remark}

We provide a brief outline of the proof of Theorem \ref{main}. According to the Yamada-Watanabe Principle (see Theorem \ref{a.3}), it suffices to prove the existence of a weak solution and uniqueness of a strong solution. Weak existence immediately follows from the continuity of coefficients of SDE \eqref{SDE} (see Theorem \ref{4.5}), and we partially follow the argument in \cite{ff3,KR} to prove the strong uniqueness. The essential ingredient to show the uniqueness of a strong solution is  a nice estimate on a solution $u$  to the Kolmogorov PDE:
\begin{align} \label{kol}
\begin{cases}
u_t +\frac{1}{2} Lu+\sum_{i=1}^m b^iZ_iu+\lambda u=f , \quad 0\leq t\leq T,\\
u(T,x)=0,
\end{cases}
\end{align}
where $L$ denotes the \it{sub-Laplacian} on $\mathbb{G}$ defined by $L:=\sum_{i=1}^m Z_i^2$. This Kolmogorov PDE \eqref{kol} appears when we apply the Zvonkin's transformation method \cite{zvo} to obtain an auxiliary SDE possessing a more regular drift. The key equation to  study PDE \eqref{kol}  is the following PDE:
\begin{align}  \label{heat equation} 
\begin{cases}
u_t - \frac{1}{2} Lu=f , \quad 0\leq t\leq T,\\
u(0,x)=0.
\end{cases}
\end{align}
Since $L$ is not elliptic and we are working with the mixed-norm parabolic Sobolev spaces of type \eqref{condition2}, we need to develop a new theory on PDE \eqref{heat equation}. We accomplish this by employing some tools from the harmonic analysis. In particular,   making use of the techniques to study subelliptic equations \cite{f2,f3} and applying the Calder\'on-Zygmund theory, we establish the well-posedness result of PDE \eqref{heat equation} in the suitable Sobolev spaces of type \eqref{condition2}. This is done in Section \ref{section 3}.

Next, we establish the well-posedness of the Kolmogorov PDE \eqref{kol} possessing singular coefficients.  The key ingredient is the Sobolev embedding theorem for the mixed-norm parabolic Sobolev spaces  (see Appendix \ref{section b}). Combining this with the result of PDE \eqref{heat equation}, we obtain a priori estimate of a solution to PDE \eqref{kol}. The well-posedness result  of  parabolic equations \eqref{kol}, involving the sub-Laplacian and singular coefficients, in the class of  mixed-norm parabolic Sobolev spaces   is also one of the main achievements of the paper. This is done in Section \ref{section 4}.

The next essential step to show the uniqueness of a strong solution to SDE \eqref{SDE} is to derive an auxiliary  SDE  from the original SDE \eqref{SDE}, which is called the Zvonkin's transformation method \cite{zvo}. This auxiliary SDE is more tractable than the original SDE \eqref{SDE} since it possesses a more regular drift. It is obtained by applying the It\^o's formula to a function $u$  which is a solution to the Kolmogorov PDE \eqref{kol}. 
 However, in our setting, it is not obvious to apply the It\^o's formula to  SDE \eqref{SDE} and a function $u$ since $u$ is not as regular. In order to overcome this problem, we establish the It\^o's formula for singular functions in Section \ref{section 5}. The main ingredient to prove this is a Krylov-type estimate  \cite{krylov}. It involves technical difficulties since we are working on the homogeneous Carnot group and SDE \eqref{SDE} involves the degenerate diffusion (see Section \ref{section 5}). 

Once the It\^o's formula for non-smooth functions is established, with the aid of an auxiliary SDE mentioned above, one can finally prove the uniqueness of a strong solution to the SDE \eqref{SDE}. We show this by controlling a difference of two strong solutions to SDE \eqref{SDE}. This is done in Section \ref{section 6}.

\section{Analysis on the nilpotent Lie group} \label{section 3}
In this section, we briefly overview the theory of analysis on the homogeneous Carnot group. We also introduce the mixed-norm parabolic Sobolev spaces and obtain their crucial properties. In particular, we establish the well-posedness result of  parabolic equations in parabolic Sobolev spaces.

\subsection{Preliminaries : homogeneous Carnot group} \label{section 3.1}
\begin{definition} \label{definition}
We say that $\mathbb{G}=(\R^N,\circ,D(\lambda))$, endowed with a Lie group structure by the composition law $\circ$, is called a \it{homogeneous group} if it is equipped with a one parameter family $\{D(\lambda)\}_{\lambda>0}$ of automorphisms of the following form
\begin{align*}
D(\lambda):(u_1,u_2,\cdots, u_N) \rightarrow (\lambda^{\alpha_1} u_1, \lambda^{\alpha_2} u_2, \cdots, \lambda^{\alpha_N} u_N)
\end{align*}
for some exponents $0<\alpha_1\leq \cdots \leq \alpha_N$. \it{Homogeneous dimension} $Q$ of $\mathbb{G}$ is defined by $Q=\alpha_1+\cdots +\alpha_N$. For $1\leq i\leq N$, let $Z_i$ be a left-invariant vector field which coincides with $\frac{\partial}{\partial x_i}$ at the origin. If Lie algebra generated by $Z_1,\cdots,Z_m$, which are 1-homogeneous left-invariant vector fields, is the whole Lie algebra $\mathfrak{g}$ of $\mathbb{G}$, then $\mathbb{G}=(\R^N,\circ,D(\lambda))$ is called a \it{homogeneous Carnot group}.
\end{definition}

If $\mathbb{G}$ is a homogeneous Carnot group, then its Lie algebra $\mathfrak{g}$ has a natural stratification (see \cite[Chapter 2.2]{sub} for details). In fact, if $\alpha_j=1$ for $j\leq m$ and  \begin{gather*}
V_1 = \text{span} (Z_1,Z_2,\cdots,Z_m),\\
V_{i+1}=[V_i,V_1],\ i>1,
\end{gather*}
 then  there exists $r$, which is called  a \it{nilpotency} of $\mathbb{G}$, satisfying
\begin{align*}
\mathfrak{g}=V_1 \oplus V_2 \oplus \cdots \oplus V_r.
\end{align*}

We assume for a moment that $\mathbb{G}=(\R^N,\circ,D(\lambda))$ is a homogeneous group with a homogeneous dimension $Q$. One can associate the \it{homogeneous norm} $\norm{\cdot}:\mathbb{G}\rightarrow \R$ to $\mathbb{G}$, smooth away from the origin, satisfying
\begin{gather*}
\norm{u} \geq 0, \quad \norm{u}=0 \Longleftrightarrow u=0, \quad \norm{D(\lambda)u}=\lambda\norm{u}.
\end{gather*}
If we denote $|\cdot|$ by a Euclidean norm, then $|x|=\mathcal{O}(\norm{x})$ as $x\rightarrow 0$. The Lebesgue measure on $\mathbb{G}=\R^N$ is a bi-invariant haar measure,  and once we make a change of coordinate $x=D(\lambda)y$,  we have $dx=\lambda^Q dy$. This implies that the homogeneous group $\mathbb{G}$ can be regarded as a homogeneous space in the sense of Coifman and Weiss \cite{cw}. This fact plays an important rule in developing a singular integral theory on the homogeneous group. 

Now, let us define the kernels of type $\alpha$ and the operators of type $\alpha$.  A function $f$ is said to be  \it{homogeneous of degree} $\alpha$ provided that for $\lambda>0$,
\begin{align*}
f(D(\lambda)x)=\lambda^\alpha f(x).
\end{align*}
\begin{definition}\cite[p. 917]{f3} \label{2.2}
$K$ is called a \it{kernel of type} $\alpha$ ($\alpha>0$) if it is smooth away from the origin and homogeneous of degree $\alpha
-Q$.
Also, $K$ is called a \it{singular integral kernel} if it is smooth away from the origin, homogeneous of degree $-Q$, and satisfies
\begin{align*}
\int_{a<|x|<b} K(x)dx=0
\end{align*}
for any $0<a<b<\infty$.   $T$ is called the \it{operator of type $\alpha$} ($0\leq \alpha<Q$) if $T$ is given by
\begin{align*}
T:f \rightarrow f*K
\end{align*} 
for some kernel $K$ of type $\alpha$. In the case $\alpha=0$, convolution is understood as a principal value sense.
\end{definition}

From now on, we assume that $\mathbb{G}=(\R^N,\circ,D(\lambda))$ is a homogeneous Carnot group with a homogeneous dimension $Q$ and nilpotency $r$. Recall that $Z_1,\cdots,Z_m$ is a (linear) basis of $V_1$, and define the \it{sub-Laplacian} $L$ by 
\begin{align*}
L=Z_1^2+\cdots+Z_m^2.
\end{align*}
According to the result by Folland \cite[Theorem 2.1]{f2}, there exists a fundamental solution of $\frac{\partial}{\partial t}-L$, which is called the heat kernel.
It turns out that the heat kernel $p$ possesses  a nice Gaussian upper bound.
\begin{theorem}\cite[Theorem 1]{jerison} \label{2.4}
For any $k\geq 0$ and indices $I=(i_1,\cdots,i_s)$ with $|I|=s \geq 0$,
\begin{align} \label{heat}
|\partial_t^k Z_Ip(t,x)| \leq C(k,I)t^{-k-\frac{s+Q}{2}} e^{-c\norm{x}^2/t}
\end{align}
holds for some constant $c$ independent of $k,s,i_1,\cdots,i_s$.
\end{theorem}
Let us now define Sobolev spaces $S^{k,p}(\mathbb{G})$ associated with vector fields $\{Z_i| 1\leq i\leq m\}$:
\begin{align*}
S^{k,p}(\mathbb{G}):=\{f|Z_If \in L^p(\mathbb{G}), |I|\leq k\},
\end{align*}  
and the associated norm $\norm{\cdot}_{S^{k,p}}$ by
\begin{align*}
\norm{f}_{S^{k,p}}=\sum_{|I|\leq k} \norm{Z_If}_{L^p}.
\end{align*}
Note that $Z_If$ is understood as a distributional sense. Like the standard Sobolev embedding theorems in the Euclidean spaces,  $S^{k,p}(\mathbb{G})$ enjoy the embedding theorems as well. Let us define Lipschitz spaces $\Gamma^\alpha(\mathbb{G})$ as follows: for $0<\alpha<1$, 
\begin{gather*}
\Gamma^\alpha(\mathbb{G}) := \Big\{f\in C_b(\mathbb{G})  \Big \vert\sup_{x,y\in \mathbb{G}}\frac{|f(x\circ y)-f(x)|}{\norm{y}^\alpha}<\infty\Big\}, \\
\Gamma^1(\mathbb{G}) := \Big\{f\in C_b(\mathbb{G})  \Big \vert \sup_{x,y\in \mathbb{G}}\frac{|f(x\circ y)+f(x\circ y^{-1})-2f(x)|}{\norm{y}}<\infty \Big\}.
\end{gather*}
For $\alpha = n+\alpha'$ with a nonnegative integer $n$ and $0<\alpha'\leq 1$, define $\Gamma^\alpha(\mathbb{G})$ by
\begin{align*}
\Gamma^\alpha(\mathbb{G}) := \{ f\in \Gamma^{\alpha'}(\mathbb{G}) \ | \  X_If\in \Gamma^{\alpha'}(\mathbb{G})  \ \text{for} \  |I|\leq n\}.
\end{align*}
We refer to \cite[Section 5]{f2} for more details about the Lipschitz spaces.

We state the Sobolev embedding theorem for the spaces $S^{k,p}(\mathbb{G})$:

\begin{theorem}\cite[Theorem 9]{f3} \label{2.8}
Suppose that $l\leq k$. Then, the space $S^{k,p}(\mathbb{G})$ is continuously embedded into $S^{l,q}(\mathbb{G})$ for $1<p<q<\infty$ satisfying $k-l=Q(\frac{1}{p}-\frac{1}{q})$. Also, the space $S^{k,p}(\mathbb{G})$ is continuously embedded into $\Gamma^\alpha(\mathbb{G})$ for $\alpha=k-\frac{Q}{p}>0$. 
\end{theorem}

\begin{remark} \label{diff}
Let us define a different version of Lipschitz spaces $\tilde{\Gamma}^\alpha(\mathbb{G})$ ($0<\alpha<1$):
\begin{gather*}
\tilde{\Gamma}^\alpha(\mathbb{G}) := \Big\{f\in C_b(\mathbb{G})  \big \vert \sup_{x,y\in \mathbb{G}}\frac{|f(y\circ x)-f(x)|}{\norm{y}^\alpha}<\infty\Big\}.
\end{gather*}
Then, it is not hard to  check that $\Gamma^\alpha_{\text{loc}}(\mathbb{G}) \subset \tilde{\Gamma}^{\alpha/r}_{\text{loc}}(\mathbb{G}) $. In fact, if we define the Euclidean  Lipschitz spaces $\Lambda^\alpha(\mathbb{G})$ ($0<\alpha<1$):
\begin{align*}
\Lambda^\alpha(\mathbb{G}) := \Big\{f\in C_b(\mathbb{G})  \Big \vert \sup_{x,y\in \mathbb{G}}\frac{|f(x+ y)-f(x)|}{|y|^\alpha}<\infty\Big\},
\end{align*}
then according to \cite[Theorem 6]{f3},
$\Gamma^\alpha_{\text{loc}}(\mathbb{G}) \subset \Lambda^{\alpha/r}_{\text{loc}}(\mathbb{G})$.
Also, from the fact $|x|=\mathcal{O}(\norm{x})$, we can deduce that $\Lambda^{\alpha/r}_{\text{loc}}(\mathbb{G})\subset \tilde{\Gamma}^{\alpha/r}_{\text{loc}}(\mathbb{G})$. Thus, we obtain $\Gamma^\alpha_{\text{loc}}(\mathbb{G}) \subset \tilde{\Gamma}^{\alpha/r}_{\text{loc}}(\mathbb{G}) $.
\end{remark}
We refer to \cite{f1,f2,fs,group} for more details of the theory of analysis on the Lie groups. 

\subsection{Mixed-norm parabolic Sobolev spaces} \label{section 3.2}
In this section, we define the mixed-norm parabolic Sobolev spaces with respect to vector fields $\{Z_i| 1\leq i\leq m\}$ and study their properties. Note that if ($\mathbb{G},\circ,D(\lambda)$) is a homogeneous group with a homogeneous dimension $Q$ and a homogeneous norm $\norm{\cdot}$, then the homogeneous group structure on $\R\times \mathbb{G}$ can be endowed as follows:  $(s,x)\tilde{\circ} (t,y)=(s+t,x\circ y)$,  $\tilde{D}(\lambda)(t,x)=(\lambda^2t,D(\lambda)x)$. Also, the homogeneous dimension of ($\R \times \mathbb{G},\tilde{\circ},\tilde{D}(\lambda)$) is equal to $Q+2$ and $\norm{(t,x)}:=\sqrt{|t|+\norm{x}^2}$ defines a  homogeneous norm on $\R\times \mathbb{G}$.

\begin{definition}
For $1\leq p,q\leq \infty$ and the integer $k\geq 0$, let us define (inhomogeneous) mixed-norm Sobolev spaces $S^{k,(q,p)}([0,T]\times \mathbb{G})$ with respect to vector fields $\{Z_i| 1\leq i\leq m\}$:
\begin{align*}
S^{k,(q,p)}([0,T]\times \mathbb{G}):=\{f\in L^q([0,T],L^p(\mathbb{G})) \ |\ Z_If \in L^q([0,T],L^p(\mathbb{G})) , |I|\leq k \}.
\end{align*}
The corresponding norm is defined by 
\begin{align*}
\norm{f}_{S^{k,(q,p)}([0,T]\times \mathbb{G})}:=\sum_{|I|\leq k} \norm{Z_If}_{L^q([0,T],L^p(\mathbb{G}))}.
\end{align*}
One can also define the homogeneous mixed-norm Sobolev spaces $\dot{S}^{k,(q,p)}([0,T]\times \mathbb{G})$  and the corresponding norm  $\norm{f}_{\dot{S}^{k,(q,p)}([0,T]\times \mathbb{G})}:= \sum_{|I|= k}  \norm{Z_If}_{L^q([0,T],L^p(\mathbb{G}))}$ similarly.
\end{definition}

From now on, we assume that $\mathbb{G}=(\R^N,\circ,D(\lambda))$ is a homogeneous Carnot group with a homogeneous dimension $Q$ and nilpotency $r$.
We first study the boundedness properties of the operators of type 0 in the mixed-norm spaces. It is the classical theory that the  operators of type 0 are bounded on $L^p(\R\times \mathbb{G})$ for $1<p<\infty$ (see \cite[p. 917]{f3}). One can generalize this result to the mixed-norm spaces $L^q(\R,L^p( \mathbb{G}))$ as an  application of the vector-valued Calder\'on-Zygmund theory.
 
\begin{theorem} \label{2.11}
Operators $T$ of type 0 are bounded in $L^q(\R,L^p( \mathbb{G}))$ for any $1<p,q<\infty$.
\end{theorem}
\begin{proof}
Let us denote $K(t,x)$ by a singular integral kernel of $T$. For $t\in \R$, let us define the operator $P_t:L^p(\mathbb{G})\rightarrow L^p(\mathbb{G})$ by
\begin{align*}
P_tf(x) = \int_{\mathbb{G}} K(t,x\circ y^{-1})f(y) dy.
\end{align*}
Then, $Tf$ can be written as
\begin{align*}
Tf(t,x)=\int_{\R} P_sf(t-s,\cdot)(x) ds.
\end{align*} 
Since the operators of type 0  are bounded in $L^p$ for $1<p<\infty$ according to \cite[p. 917]{f3}, $T$ is bounded in $L^p(\R,L^p(\mathbb{G}))$. In order to extend this to the general cases $q\neq p$, it suffices to prove the following inequality: for some constant $c>0$ independent of $s$,
\be \label{calderon}
\int_{|t|\geq c|s|} \norm{P_t-P_{t-s}}_{L^p\rightarrow L^p}dt \leq C < \infty,
\ee
according to the vector-valued Calder\'on-Zygmund theory. One can represent $P_t-P_{t-s}$  in terms of the singular integral kernel $K$:
\begin{align} \label{2110}
(P_t-P_{t-s})f(x)=\int_{\mathbb{G}} [K(t,x\circ y^{-1})-K(t-s,x\circ y^{-1})]f(y)dy.
\end{align}
Let us define a homogeneous norm $\norm{\cdot}$ on $\R\times \mathbb{G}$  by $\norm{(t,x)}:=\sqrt{|t|+\norm{x}^2}$.  Since the singular integral kernel $K(t,x)$ is homogeneous of degree $-(Q+2)$, using \cite[Lemma 5.2]{kv}, there exist $C,\delta>0$ such that whenever $C\norm{(s,y)} \leq \norm{(t,x)}$,
\begin{align*}
|K((t,x)\circ (s,y))-K(t,x)| +|K((s,y)\circ (t,x))-K(t,x)| \leq C\frac{\norm{(s,y)}^\delta}{\norm{(t,x)}^{Q+2+\delta}}.
\end{align*}
 Therefore, for some constant $C_1$ depending  on $\delta$, whenever $|t|>C^2|s|$, 
\begin{align*}
\norm{K(t,x)-K(t-s,x)}_{L^1(\mathbb{G})}&\leq \norm{\frac{|s|^{\delta/2}}{(|t|+\norm{x}^2)^{\delta/2+(Q+2)/2}} }_{L^1(\mathbb{G})} \\
&=|s|^{\delta/2}\int_{\mathbb{G}} \frac{|t|^{Q/2}}{|t|^{\delta/2+(Q+2)/2}(1+\norm{z}^2)^{\delta/2+(Q+2)/2}}dz  = C_1\frac{|s|^{\delta/2}}{|t|^{\delta/2+1}}.
\end{align*}
Thus, applying the Young's convolution inequality to \eqref{2110},
\begin{align*}
\norm{P_t-P_{t-s}}_{L^p\rightarrow L^p} \leq  C_1\frac{|s|^{\delta/2}}{|t|^{\delta/2+1}}.
\end{align*} 
Integrating this in $t$, we obtain
\begin{align*}
\int_{|t|\geq C^2|s|} \norm{P_t-P_{t-s}}_{L^p\rightarrow L^p}dt \leq  C_1\int_{|t|\geq C^2|s|}\frac{|s|^{\delta/2}}{|t|^{\delta/2+1}} dt \leq \frac{4C_1}{\delta C^\delta },
\end{align*} 
which immediately implies \eqref{calderon}.
\end{proof}

We now focus on the following parabolic equation, involving the sub-Laplacian $L$:
\be \label{hypo heat}
\begin{cases}
u_t-L u=f, \quad 0\leq t\leq T, \\
u(0,x)=0.
\end{cases}
\ee 
 Since \eqref{hypo heat} serves as a toy equation to study the Kolmogorov PDE  \eqref{kol}, it is crucial to establish the well-posedness result of the equation \eqref{hypo heat} in the mixed-norm  parabolic Sobolev spaces. We introduce the auxiliary function spaces:  for $1\leq p,q\leq \infty$ and $k\geq 2$, $u$ belongs to the function space $\tilde{S}^{k,(q,p)}([0,T]\times \mathbb{G})$ when
\begin{align} \label{til}
u\in S^{k,(q,p)}([0,T]\times \mathbb{G}), \quad
u_t\in S^{k-2,(q,p)}([0,T]\times \mathbb{G}).
\end{align}
The corresponding norm $\norm{\cdot}_{\tilde{S}^{k,(q,p)}}$ is defined by
\begin{align*}
\norm{u}_{\tilde{S}^{k,(q,p)}([0,T]\times \mathbb{G})}:=\norm{u}_{S^{k,(q,p)}([0,T]\times \mathbb{G})}+\norm{u_t}_{S^{k-2,(q,p)}([0,T]\times \mathbb{G})}.
\end{align*}
One can prove the well-posedness result of the PDE \eqref{hypo heat} in the class $\tilde{S}^{k,(q,p)}([0,T]\times \mathbb{G})$:

\begin{theorem} \label{3.1}
Suppose that $1<p,q<\infty$. Then,  for any $f\in S^{k,(q,p)}([0,T]\times \mathbb{G})$, there exist a unique solution $u\in \tilde{S}^{k+2,(q,p)}([0,T]\times \mathbb{G})$ to PDE  \eqref{hypo heat}. Also, there exists some constant $C$ independent of $T$ such that for any $f\in S^{k,(q,p)}([0,T]\times \mathbb{G})$,
\be \label{a priori 0}
\norm{u}_{\dot{S}^{k+2,(q,p)}([0,T]\times \mathbb{G})}\leq C\norm{f}_{\dot{S}^{k,(q,p)}([0,T]\times \mathbb{G})},
\ee
\be \label{a priori 0'}
\norm{u}_{\tilde{S}^{k+2,(q,p)}([0,T]\times \mathbb{G})}\leq C\max \{T,1\}\norm{f}_{S^{k,(q,p)}([0,T]\times \mathbb{G})}.
\ee
\end{theorem} 
\begin{proof}
Throughout the proof, we use a notation $P$  for the heat kernel in order to avoid a confusion with the exponent $p$. Also, $P_t$ denotes a semigroup generated by the sub-Laplacian $L$. For test functions $f\in C^\infty_c(\R \times \mathbb{G})$, let us define
\begin{align}\label{q}
Qf(t,x):=\int_{-\infty}^t P_{t-s}f(s)(x)ds.
\end{align}
Note that  $u:=Qf$ is a classical solution to the equation $u_t-Lu=f$.

Step 1. A priori estimate on $\norm{Qf}_{\dot{S}^{k+2,(q,p)}([0,T]\times \mathbb{G})}$: let us first prove that
\begin{align} \label{higher}
\norm{Qf}_{\dot{S}^{k+2,(q,p)}(\R \times \mathbb{G})} \leq C\norm{f}_{\dot{S}^{k,(q,p)}(\R \times \mathbb{G})}
\end{align}
for any test functions $f$. In the case  of  $k=0$, Theorem \ref{2.11} immediately implies \eqref{higher}. In fact, for any $1\leq i_1,i_2\leq m$, we have a representation formula:
\begin{align*}
Z_{i_1}Z_{i_2}Qf = f* Z_{i_1}Z_{i_2} P
\end{align*}
(convolution acts on $\R \times \mathbb{G}$), and note that $Z_{i_1}Z_{i_2} P$ is a singular kernel.

Now, let us prove \eqref{higher} when $k=1$. We make use of the  arguments in the theory of  subelliptic estimates (see for instance \cite[Section 3.3.5]{b}).
If we denote $Z^R_i$ ($1\leq i\leq N$) by a right-invariant vector field which coincides with $\frac{\partial}{\partial x_i}$ at the origin, then we have
\begin{align*}
(Z_i f)*g=f*(Z^R_ig)
\end{align*}
(convolution acts on $\mathbb{G}$). Also, for each $1\leq i\leq m$, there exist homogeneous functions $\beta_{ji}$ of degree $\alpha_j-1$ ($1\leq j\leq N$) such that
\begin{align} \label{b.5}
Z_iu=\sum_{j=1}^N Z^R_{j} (\beta_{ji}u)
\end{align}
holds for all test functions $u$  on $\mathbb{G}$ (see \cite[p.64]{b}). Note that
since each $Z_j$ ($1\leq j\leq N$) can be written as a commutator of $Z_i$'s ($1\leq i\leq m$) with order $\alpha_j$ and the following identity 
\begin{align*}
[Z_i^R,Z_j^R]=-[Z_i,Z_j]^R
\end{align*}
holds for any indices $i,j$, we can write $Z_j^R$ as
\begin{align*}
Z_j^R = \sum_{l,I} Z_{jl}^R  Z^R_{jI}.
\end{align*}
Here, each $Z_{jl}$ is one of $Z_i$'s ($1\leq i\leq m$), and each $Z_{jI}^R$ is of the form $Z_{s_1}^R\cdots Z_{s_{\alpha_j-1}}^R$ for $1\leq s_1,\cdots,s_{\alpha_j-1} \leq m$. 
Therefore, applying this to \eqref{b.5}, for any indices $1\leq i_2,i_3\leq m$,
\begin{align*}
f * Z_{i_2}Z_{i_3}P &= f *(\sum_{j=1}^N Z^R_{j} (\beta_{ji_2} Z_{i_3}P)) =\sum_{j=1}^N \sum_{l,I} Z_{jl}f * (Z^R_{jI} (\beta_{ji_2} Z_{i_3}P))
\end{align*}
(convolution acts on $\R \times \mathbb{G}$). Differentiating this in $Z_{i_1}$ ($1\leq i_1\leq m$) direction,
\begin{align*}
Z_{i_1}Z_{i_2}Z_{i_3}u = \sum_{j=1}^N \sum_{l,I} Z_{jl}f * (Z_{i_1}Z^R_{jI} (\beta_{ji_2} Z_{i_3}P)).
\end{align*}
Recall that $P$ is a kernel of type 2, $\beta_{ji_2}$ is homogeneous of degree $\alpha_j-1$,  $Z^R_{jI}$ is a differential operator of order $\alpha_j-1$, and $Z_{i_2}$, $Z_{i_3}$ are differential operators of order 1. From this, it follows that
$Z_{i_1}Z^R_{jI} (\beta_{ji_2} Z_{i_3}P)$
is a singular integral kernel. Since  the  operators of type 0 are bounded in $L^q(\R, L^p(\mathbb{G}))$ according to Theorem \ref{2.11},  we have
\begin{align*}
\norm{Z_{i_1}Z_{i_2}Z_{i_3}u}_{L^q(\R, L^p(\mathbb{G}))} \leq C \sum_{1\leq i\leq m} \norm{Z_if}_{L^q(\R, L^p(\mathbb{G}))}.
\end{align*} This concludes the proof when $k=1$. Similar arguments work for general $k$ as well.

Step 2. A priori estimate on  $\norm{Qf}_{\tilde{S}^{k+2,(q,p)}([0,T]\times \mathbb{G})}$: we prove that
\begin{align} \label{3131}
\norm{Qf}_{\tilde{S}^{k+2,(q,p)}([0,T]\times \mathbb{G})}\leq C\max \{T,1\}\norm{f}_{S^{k,(q,p)}([0,T]\times \mathbb{G})}
\end{align} 
for any test functions $f$. Since $u(t):=Qf(t)$ with $0\leq t\leq T$ depends only on  $f(s)$ with $s\leq t$, according to the estimate \eqref{higher}, for any $0\leq l\leq k$,
\begin{align*}
\norm{u}_{\dot{S}^{l+2,(q,p)}([0,T]\times \mathbb{G})}\leq C\norm{f}_{\dot{S}^{l,(q,p)}([0,T]\times \mathbb{G})}.
\end{align*}
Note that the constant $C$ can be chosen independently of $T$ due to the existence of scaling  $u(t,x) \mapsto u(\lambda^2t,D(\lambda) x)$
for  $\lambda>0$. Summing these inequalities over $0\leq l\leq k$,
\begin{align}\label{3132}
\sum_{l=0}^k \norm{u}_{\dot{S}^{l+2,(q,p)}([0,T]\times \mathbb{G})}\leq C\norm{f}_{S^{k,(q,p)}([0,T]\times \mathbb{G})}.
\end{align}
From the equation $u_t-Lu=f$ and the estimate \eqref{3132}, we have
\begin{align} \label{3135}
\norm{u_t}_{S^{k,(q,p)}([0,T]\times \mathbb{G})} \leq C \norm{f}_{S^{k,(q,p)}([0,T]\times \mathbb{G})}.
\end{align}
Applying this to  the trivial inequality $u(t,x)\leq \int_0^T |u_t(s,x)|ds$, we obtain
\begin{align} \label{3133}
\norm{u}_{L^q([0,T],L^p(\mathbb{G}))} \leq T\norm{u_t}_{L^q([0,T],L^p(\mathbb{G}))}\leq CT \norm{f}_{S^{k,(q,p)}([0,T]\times \mathbb{G})}.
\end{align} 
 Also, the interpolation type inequality  \cite[Theorem 3.3]{interpolation}
allows us to obtain 
\begin{align} \label{3134}
\norm{u}_{\dot{S}^{1,(q,p)}([0,T]\times \mathbb{G})}\leq C\max\{T,1\}\norm{f}_{S^{k,(q,p)}([0,T]\times \mathbb{G})}.
\end{align}  Thus, using \eqref{3132}, \eqref{3135}, \eqref{3133}, and \eqref{3134}, we obtain \eqref{3131}. 

Step 3. Existence of a solution:  it can be proved by a standard approximation argument thanks to the estimate \eqref{3131}. Also, \eqref{a priori 0} and \eqref{a priori 0'} hold for any $f\in S^{k,(q,p)}([0,T]\times \mathbb{G})$.

Step 4. Uniqueness of a solution: it suffices to prove that if $u\in \tilde{S}^{k+2,(q,p)}([0,T]\times \mathbb{G})$ is a solution to the equation \eqref{hypo heat} with $f=0$, then $u=0$. Choose the approximation $u_n$, each of which is smooth and has compact support, converging to $u$ in $\tilde{S}^{k+2,(q,p)}([0,T]\times \mathbb{G})$ norm. It follows that
\begin{align*}
\norm{(u_n)_t-Lu_n}_{S^{k,(q,p)}} \rightarrow 0
\end{align*}
as $n\rightarrow \infty$. Since $u_n = Q((u_n)_t-Lu_n)$, according to the estimate \eqref{3131}, we have
\begin{align*}
\norm{u_n}_{\tilde{S}^{k+2,(q,p)}([0,T]\times \mathbb{G})} \rightarrow 0
\end{align*}
as $n\rightarrow \infty$. Thus, $\norm{u}_{\tilde{S}^{k+2,(q,p)}([0,T]\times \mathbb{G})}=0$, which concludes the proof.
\end{proof}
One can also derive the Sobolev embedding theorems for the spaces $S^{k,(q,p)}([0,T]\times \mathbb{G})$. We refer to Appendix \ref{section b} for the statement and  proof. The Sobolev embedding Theorem \ref{2.13} is a key ingredient to prove the well-posedness of the Kolmogorov PDE, which will be done in  Section \ref{section 4}.

\section{Kolmogorov PDE results} \label{section 4}

In this section, we establish the well-posedness result of the following Kolmogorov PDE:
\be \label{Kolmogorov PDE}
\begin{cases}
u_t - \frac{1}{2} Lu+\sum_{i=1}^m b^iZ_iu+\lambda u=f , \quad 0\leq t\leq T,\\
u(0,x)=0.
\end{cases}
\ee
on the homogeneous Carnot group $\mathbb{G}$ for singular functions $b$, $f$ and $\lambda \in \R$. The solution $u$ to Kolmogorov PDE \eqref{Kolmogorov PDE} plays a crucial role in proving the uniqueness of a strong solution to SDE \eqref{SDE}. In fact, this PDE appears when we apply the Zvonkin's transformation method \cite{zvo} to obtain an auxiliary SDE.  From now on, for any Banach spaces $X$, let us define 
\begin{align*}
\norm{b}_X:=\sum_{i=1}^m \norm{b_i}_X.
\end{align*}

\begin{theorem} \label{3.3}
Assume that $b$ satisfies the conditions \eqref{condition1} and \eqref{condition2} for exponents $p$ and $q$ satisfying \eqref{pq}. Then, for any $f \in S^{r-1,(q,p)}([0,T]\times \mathbb{G})$, there exists a unique solution $u\in \tilde{S}^{r+1,(q,p)}([0,T]\times \mathbb{G})$ to PDE \eqref{Kolmogorov PDE}. Furthermore, we have the following  estimate:
\be \label{a priori for PDE}
\norm{u}_{S^{r+1,(q,p)}([0,T]\times \mathbb{G})}+\norm{u_t}_{S^{r-1,(q,p)}([0,T]\times \mathbb{G})}\leq C(b,\lambda) \norm{f}_{S^{r-1,(q,p)}([0,T]\times \mathbb{G})}
\ee
\end{theorem}

\begin{proof}
Let us first prove an a priori estimate \eqref{a priori for PDE}. For $0\leq t\leq T$, let us define
\begin{align*}
I(t)=\norm{u}^q_{{S}^{r+1,(q,p)}([0,t]\times \mathbb{G})}+\norm{u_t}^q_{{S}^{r-1,(q,p)}([0,t]\times \mathbb{G})}.
\end{align*}
Then, using the estimate \eqref{a priori 0'}, we have
\begin{align} \label{331}
I(t) &\leq C \norm{\sum_{i=1}^m b^iZ_iu+\lambda u+f}_{S^{r-1,(q,p)}([0,t]\times \mathbb{G})}^q.
\end{align}
Since $p$ and $q$ satisfy  \eqref{pq}, according to the parabolic Sobolev embedding Theorem \ref{2.13},
\begin{align*}
\norm{u}_{S^{r,(\infty,\infty)}([0,t]\times \mathbb{G})}^q \leq C (\norm{u}_{S^{r+1,(q,p)}([0,t]\times \mathbb{G})}^q+\norm{u_t}_{S^{r-1,(q,p)}([0,t]\times \mathbb{G})}^q)=CI(t).
\end{align*}
Therefore, one can deduce that for each $1\leq i\leq m$,
\begin{align} \label{3310}
&\norm{b^iZ_iu}_{S^{r-1,(q,p)}([0,t]\times \mathbb{G})}^q=\int_0^t \norm{b^iZ_iu(s)}_{S^{r-1,p}( \mathbb{G})}^q ds \nonumber \\
&\leq \int_0^t \norm{b^i(s)}_{S^{r-1,p}( \mathbb{G})}^q\norm{Z_iu(s)}_{L^\infty( \mathbb{G})}^q ds \leq C\int_0^t \norm{b^i(s)}_{S^{r-1,p}( \mathbb{G})}^qI(s)ds.
\end{align}
Also, using Minkowski's integral inequality,
\begin{align}\label{3311}
\norm{\lambda u}_{S^{r-1,(q,p)}([0,t]\times \mathbb{G})}^q  ds \leq   C\lambda ^q \int_0^t \Big[ \int_0^s \norm{u_t(l)}_{S^{r-1,p}(\mathbb{G})}^q dl\Big]ds \leq  C \lambda^q \int_0^t I(s)ds.
\end{align}
Therefore, applying \eqref{3310} and \eqref{3311} to \eqref{331}, 
\begin{align*}
I(t)\leq C\int_0^t (\norm{b(s)}_{S^{r-1,p}( \mathbb{G})}^q+\lambda^q  )I(s)ds+C\norm{f}_{S^{r-1,(q,p)}([0,T]\times \mathbb{G})}.
\end{align*}
Using the Gr\"onwall's inequality, for each $0\leq t\leq T$,
\begin{align*}
I(t)\leq  C\norm{f}_{S^{r-1,(q,p)}([0,T]\times \mathbb{G})}\exp\big[C\norm{b}_{S^{r-1,(q,p)}([0,t]\times \mathbb{G})}^q + Ct\lambda^q\big].
\end{align*}
In particular, the case $t=T$ implies \eqref{a priori for PDE}.

Once a priori estimate \eqref{a priori for PDE} is obtained, the existence and uniqueness of a solution to the PDE \eqref{Kolmogorov PDE}  immediately follows from  the standard method of continuity (see for example \cite[Theorem 4.7]{f}). 
\end{proof}

Assume that $b$ satisfies the conditions \eqref{condition1}, \eqref{condition2}  for the exponents $p,q$ satisfying \eqref{pq}. Also, suppose that a function $f\in S^{r-1,(q,p)}([0,T]\times \mathbb{G})$ taking values in $\mathbb{G}=\R^N$ is given. This means that each Euclidean coordinate of $f$ belongs to $S^{r-1,(q,p)}([0,T]\times \mathbb{G})$. Let us now consider the following PDE:
\be \label{transform PDE}
\begin{cases}
u_t +\frac{1}{2} Lu+\sum_{i=1}^m b^iZ_iu -\lambda u= f, \quad 0\leq t\leq T, \\
u(T,x)=0.
\end{cases}
\ee
 $u$ being a solution to PDE \eqref{transform PDE} means that \eqref{transform PDE} holds in each Euclidean coordinate. According to Theorem \ref{3.3}, by reversing time, one can deduce that PDE \eqref{transform PDE} has a (unique)	 solution $\tilde{u}^{\lambda}\in \tilde{S}^{r+1,(q,p)}([0,T]\times \mathbb{G})$ taking values in $\mathbb{G}=\R^N$. We introduce an auxiliary function $\Phi^\lambda$ in the next proposition, which plays a crucial role in Section \ref{section 6}.

\begin{proposition} \label{3.5}
There exist an open set $\Omega$ containing $x_0$,  $\lambda\in \R$, and a version $u^\lambda$ of $\tilde{u}^\lambda$ such that $\Phi^\lambda(t,x):=x+u^\lambda(t,x)$ satisfies the following properties: \\
(i) $\Phi^{\lambda}$ is continuous in $(t,x)$ and $\Phi^{\lambda}(t,\cdot)$ is $C^1$ for each $0\leq t\leq T$. \\
(ii) $\Phi^{\lambda}(t,\cdot)$ is a $C^1$ diffeomorphism from $\Omega$ onto its image for each $0\leq t\leq T$.\\
(iii) For each $0\leq t\leq T$,
\begin{align*}
\frac{1}{2}\leq \norm{\nabla \Phi^{\lambda}(t,\cdot)}_{L^\infty(\Omega)}\leq 2, \quad \frac{1}{2}\leq \norm{\nabla (\Phi^{\lambda})^{-1}(t,\cdot)}_{L^\infty(\Phi^{\lambda}(t,\Omega))}\leq 2.
\end{align*} 
\end{proposition}
\begin{proof}
Throughout the proof, in order to alleviate the notations,  we use  $\tilde{u}$ and $\Phi$ instead of $\tilde{u}^{\lambda}$ and $\Phi^{\lambda}$, respectively.

Step 1. Proof of the property (i): let us prove that  there exists a continuous version $u$ of $\tilde{u}$ such that $u(t,\cdot)$ is $C^1$ for each $t$. It suffices to show that for the arbitrary bounded and open set $U$ in $\R^N$, there exists a version $u$ of $\tilde{u}$ such that $u$ is continuous on $[0,T]\times U$ and $u(t,\cdot)\in C^1(U)$. Choose a smooth approximation $u_n$ converging to $\tilde{u}$ in  $\tilde{S}^{r+1,(q,p)}([0,T]\times \mathbb{G})$ norm. According to Theorem \ref{2.13}, for any indices $|I|\leq r$,
\begin{align}\label{35200}
\norm{Z_I (u_n-u_m)}_{L^\infty([0,T]\times \mathbb{G})} \leq C\norm{u_n-u_m}_{\tilde{S}^{r+1,(q,p)}([0,T]\times \mathbb{G})}.
\end{align}
Since each standard vector field on $\R^N$ can be written as commutators of $Z_i$'s up to order $r$, it follows from \eqref{35200} that 
\begin{align*}
\norm{\nabla(u_n-u_m)}_{L^\infty([0,T]\times U)} \leq C \norm{u_n-u_m}_{\tilde{S}^{r+1,(q,p)}([0,T]\times \mathbb{G})}.
\end{align*}
for some constant $C=C(U)$. This implies that there exists $w\in C_b([0,T]\times U)$ such that 
\begin{align}  \label{3500}
\norm{w-\nabla u_n}_{L^\infty([0,T] \times U)} \rightarrow 0
\end{align} 
as $n\rightarrow 0$.  Also, since the sequence $\{u_n\}$ is Cauchy in $L^\infty([0,T]\times \mathbb{G})$ norm by Theorem \ref{2.13}, there exists $u\in C_b([0,T]\times \mathbb{G})$, which is a version of $\tilde{u}$, such that as $n\rightarrow \infty$,
\begin{align} \label{3501}
\norm{u-u_n}_{L^\infty([0,T] \times \mathbb{G})} \rightarrow 0.
\end{align} 
Thanks to \eqref{3500} and \eqref{3501}, for each $t$, $u(t,\cdot)$ is $C^1$ on $U$ and its spatial  derivative is $w(t,\cdot)$. 

Step 2. Estimate on $\norm{u}_{S^{r,(\infty,\infty)}}$: from now on, we denote $u$ by a function selected in the Step 1. We now claim that for arbitrary $\epsilon>0$, there exists a sufficiently large $\lambda$ such that
\be \label{3.5claim}
\norm{Z_Iu}_{L^\infty([0,T]\times \mathbb{G})} \leq \epsilon
\ee 
holds for all indices $|I|\leq r$. 
We have the following representation formula for $u$:
\begin{align*}
u(t)=\int_t^T e^{-\lambda (s-t)}P_{s-t}(f+\sum_{i=1}^m b^iZ_iu)(s) ds
\end{align*}
(see for example \cite[Lemma 3.4]{ff3}).
Differentiating this in $Z_I$ ($|I|\leq r$) directions, 
\be \label{3.5rep}
Z_Iu(t)=\int_t^T e^{-\lambda (s-t)}Z_IP_{s-t}(f+\sum_{i=1}^m b^iZ_iu)(s)ds.
\ee 
Note that for  $g\in S^{r-1,p}(\mathbb{G})$, according to Lemma \ref{lemma b.2},
\begin{align*}
\norm{Z_IP_tg}_{L^\infty(\mathbb{G})}&=\norm{g*Z_Ip_t}_{L^\infty(\mathbb{G})} \leq Ct^{-(\frac{Q}{2p}+\frac{1}{2})}\norm{g}_{S^{r-1,p}(\mathbb{G})}
\end{align*}
for each $0\leq t\leq T$ and any indices $|I|\leq r$. Therefore, applying this to \eqref{3.5rep}, we have
\begin{align*}
&\sum_{|I|\leq r} \norm{Z_Iu(t)}_{L^\infty(\mathbb{G})} \leq \sum_{|I|\leq r}\int_t^T e^{-\lambda (s-t)}\norm{Z_IP_{s-t}(f+\sum_{i=1}^m b^iZ_iu)(s)}_{L^\infty(\mathbb{G})} ds \\
&\leq  C\int_t^T e^{-\lambda (s-t)}(s-t)^{-(\frac{Q}{2p}+\frac{1}{2})}\norm{(f+\sum_{i=1}^m b^iZ_iu)(s)}_{S^{r-1,p}(\mathbb{G})} ds \\
&\leq C\int_t^T e^{-\lambda (s-t)}(s-t)^{-(\frac{Q}{2p}+\frac{1}{2})}\Big(\norm{f(s)}_{S^{r-1,p}(\mathbb{G})}+\norm{b(s)}_{S^{r-1,p}(\mathbb{G})}\sum_{|I|\leq r}\norm{Z_Iu(s)}_{L^\infty(\mathbb{G})}\Big) ds.
\end{align*}
Using the modified version of Gr\"onwall's inequality (see \cite[Lemma 3.1]{ff3}), we obtain
\be \label{gronwall}
\sum_{|I|\leq r}\norm{Z_Iu(t)}_{L^\infty(\mathbb{G})} \leq \alpha(t)+\int_t^T \alpha(s)\beta_t(s)\exp\Big[\int_t^s\beta_t(l)dl\Big]ds,
\ee 
where $\alpha(s)$ and $\beta_t(s)$ are defined by 
\begin{gather*}
\alpha(s)=C\int_s^T e^{-\lambda (l-s)}(l-s)^{-(\frac{Q}{2p}+\frac{1}{2})}\norm{f(l)}_{S^{r-1,p}(\mathbb{G})}dl, \\
\beta_t(s)=Ce^{-\lambda (s-t)}(s-t)^{-(\frac{Q}{2p}+\frac{1}{2})}\norm{b(s)}_{S^{r-1,p}(\mathbb{G})}.
\end{gather*}
If we denote $q'$ by a conjugate exponent of $q$, then by H\"older's inequality,
\begin{align*}
\alpha(t)&=C\int_t^T e^{-\lambda (s-t)}(s-t)^{-(\frac{Q}{2p}+\frac{1}{2})}\norm{f(s)}_{S^{r-1,p}(\mathbb{G})}ds \\
&\leq C\Big[\int_0^T e^{-q'\lambda s}s^{-q'(\frac{Q}{2p}+\frac{1}{2})}ds\Big]^{1/q'}\norm{f}_{S^{r-1,(q,p)}([0,T]\times \mathbb{G})}.
\end{align*}
Since the condition \eqref{pq} implies that $q'(\frac{Q}{2p}+\frac{1}{2})<1$,
one can easily check that 
\begin{align*}
\lim_{\lambda \rightarrow \infty} \int_0^T e^{-q'\lambda s}s^{-q'(\frac{Q}{2p}+\frac{1}{2})}ds = 0,
\end{align*} which implies that
\begin{align} \label{4200}
\lim_{\lambda \rightarrow \infty} \Big[\sup_{0\leq t\leq T} \alpha(t)\Big] = 0.
\end{align} 
Similarly, applying the H\"older's inequality as above,
\begin{align}\label{4201}
\int_t^T \beta_t(s)ds \leq C\norm{b}_{S^{r-1,(q,p)}([0,T]\times \mathbb{G})}<\infty.
\end{align}
Therefore,  \eqref{gronwall}, \eqref{4200}, and \eqref{4201} imply that for sufficiently large $\lambda$, \eqref{3.5claim} holds for all indices $|I|\leq r$.

Step 3. Proof of the properties (ii) and (iii): since each standard vector field on $\R^N$ can be written as a linear combination of commutators of $Z_i$'s with order $\leq r$, for any bounded set $U$ in $\R^N$ containing $x_0$, there exists a constant $C=C(U)$ satisfying
\begin{align*}
\norm{\nabla u}_{L^\infty([0,T]\times U)} \leq C \norm{u}_{S^{r,(\infty,\infty)}([0,T]\times U)}.
\end{align*} 
Therefore, thanks to the claim \eqref{3.5claim} proved in  Step 2, for sufficiently large $\lambda$, we have
\begin{align}\label{end}
\norm{\nabla u}_{L^\infty([0,T]\times U)} \leq \frac{1}{2},
\end{align}
which immediately implies the first inequality in the condition (iii).
Since $\nabla \Phi(t,\cdot)$ is continuous and non-singular on $U$, there exists an open set $\Omega \subset U$ containing $x_0$ such that $\Phi(t,\cdot)$ is $C^1$ diffeomorphism from $\Omega$ onto its image according to the inverse function theorem. 
Also, using \eqref{end} and the identity
\begin{align*}
\nabla\Phi^{-1}(t,x)=[\nabla\Phi(t,\Phi^{-1}(t,x))]^{-1}=[I+\nabla u(t,\Phi^{-1}(t,x))]^{-1},
\end{align*} 
we obtain the second inequality in the condition (iii).
This concludes the proof. 
\end{proof}

\begin{remark} \label{remark36}
Since we assumed  that $b^i$'s have compact support (see Remark \ref{remark1.3}), and $Z_i$'s are smooth vector fields,  each Euclidean coordinate of $b$ belongs to $S^{r-1,(q,p)}([0,T]\times \mathbb{G})$. Thus, Proposition \ref{3.5} is  applicable for $f=b$. In fact, there exists $u\in \tilde{S}^{r+1,(q,p)}([0,T]\times \mathbb{G})$ taking values in $\R^N$, which is $C^1$ in $x$,   satisfying
\begin{align} \label{314}
\begin{cases}
u_t + \frac{1}{2}Lu+\sum_{i=1}^m b^iZ_iu -\lambda u = -b,\quad 0\leq t\leq T, \\
u(T,x)=0.
\end{cases}
\end{align}
Also, there exist $\lambda\in \R$ and an open set $\Omega$ in $\R^N$ containing $x_0$ such that  $\Phi(t,x)=x+u(t,x)$ satisfies  (i), (ii), and (iii) in Proposition \ref{3.5}. From now on, we use these notations $u$, $\Phi$, and $\Omega$. Finally, we choose versions of $u_t$, $Z_iu$ ($1\leq i\leq m$), $Lu$  such that $t$-a.e., $u_t(t,\cdot)$, $Z_iu(t,\cdot)$ ($1\leq i\leq m$), $Lu(t,\cdot)$ are continuous. In fact,  $u_t,Z_iu,Lu\in S^{r-1,(q,p)}([0,T]\times \mathbb{G})$, which implies that  $u_t(t,\cdot)$, $Z_iu(t,\cdot)$, $Lu(t,\cdot)\in S^{r-1,p}(\mathbb{G})$ for almost every $t$, and thus such versions can be obtained according to Theorem \ref{2.8}. Since  the left hand side of \eqref{314} and $b(t,\cdot)$ are both continuous in $x$ for $t$-a.e (see Remark \ref{remark1.3}), it follows that $t$-a.e, the equation \eqref{314} is satisfied for every $x\in \mathbb{G}$.
\end{remark}

\section{It\^o's formula for singular functions} \label{section 5}
 In order to prove the strong uniqueness of SDE \eqref{SDE}, we use the Zvonkin's transformation method \cite{zvo} to obtain an auxiliary SDE. As mentioned in Section \ref{section 2}, this auxiliary SDE is more tractable than the original SDE \eqref{SDE} since it possesses a more regular drift coefficient. When we use the Zvonkin's transformation method, a function to which we apply the It\^o's formula is not as regular.  In order to overcome this problem, we need to  establish the It\^o's formula for a large class of singular functions.

The key ingredient to obtain the It\^o's formula for non-smooth functions is a Krylov-type estimate \cite{krylov}. This type of estimate has been used successfully to prove the well-posedness of a singular SDE with the non-degenerate noise (see for example \cite{zhang1}).
In next proposition, we establish a Krylov-type estimate for the degenerate diffusion case \eqref{SDE}. Since we are working on the homogeneous Carnot group and the SDE \eqref{SDE} possesses the degenerate diffusion, the proof involves some technical difficulties.

\begin{proposition} \label{4.0}
Assume that  $b$ satisfies the conditions \eqref{condition1}, \eqref{condition2} for the exponents $p,q$ satisfying \eqref{pq}.
Suppose that $X_t$ is a solution to SDE \eqref{SDE}. Then, for each $0\leq s\leq t\leq T$, the estimate
\be \label{itotanaka}
\E \Big[\int_s^t f(r,X_r) dr \Big \vert \mathcal{F}_s\Big]\leq  C(t-s)^{1-(\frac{2}{q}+\frac{Q}{p})}\norm{f}_{L^{q/2}([0,t], L^{p/2}(\mathbb{G}))}
\ee 
holds for any $f\in L^{q/2}([0,t], L^{p/2}(\mathbb{G}))$  such that $f(r,\cdot)$ is continuous for a.e. $r\in [0,t]$. Here, a constant $C$ is independent of $s,t$, and a function $f$.
\end{proposition}
\begin{proof}
It suffices to prove \eqref{itotanaka} for non-negative $f\in L^{q/2}([0,t], L^{p/2}(\mathbb{G}))$ such that $f(r,\cdot)$ is continuous for a.e. $r\in [0,t]$.

Step 1. The auxiliary PDE result: let us prove that  for any $f\in L^{q/2}([0,t], L^{p/2}(\mathbb{G}))$, one can find a solution $w\in \tilde{S}^{2,(q/2,p/2)}([0,t]\times \mathbb{G})$ to the equation:
\be \label{auxiliary PDE}
\begin{cases}
w_t + \frac{1}{2} Lw+\sum_{i=1}^m b^iZ_iw=f, \quad \text{in} \ [0,t]\times \mathbb{G}, \\
w(t,x)=0,
\end{cases}
\ee
satisfying that for some constant $C$,
\be \label{04}
\norm{w}_{\tilde{S}^{2,(q/2,p/2)}([0,t]\times \mathbb{G})} \leq C\norm{f}_{L^{q/2}([0,t],L^{p/2}(\mathbb{G}))}.
\ee  
For $u\in \tilde{S}^{2,(q/2,p/2)}([0,T]\times \mathbb{G})$ (see \eqref{til} for definition), let us consider the following PDE:
\be \label{34}
\begin{cases}
w_t+\frac{1}{2}Lw=f - \sum_{i=1}^m b^iZ_iu,\quad  \text{in} \ [0,t]\times \mathbb{G}, \\
w(t,x)=0.
\end{cases}
\ee
Note that according to H\"older's inequality and the parabolic Sobolev embedding Theorem \ref{2.13},  for each  $1\leq i\leq m$,
\begin{align*}
\norm{b^iZ_iu}_{L^{q/2}([0,t],L^{p/2}(\mathbb{G}))} &\leq \norm{b^i}_{L^q([0,t],L^p(\mathbb{G}))} \norm{Z_iu}_{L^{ q}([0,t],L^{ p}(\mathbb{G}))}  \\
&\leq CT^{\frac{1}{2}[1-(\frac{2}{q}+\frac{Q}{p})]}\norm{b^i}_{L^q([0,t],L^p(\mathbb{G}))} \norm{u}_{\tilde{S}^{2,(q/2,p/2)}([0,t]\times \mathbb{G})}.
\end{align*}
Therefore, the right hand side  of  PDE \eqref{34} belongs to $L^{q/2}([0,T],L^{p/2}(\mathbb{G}))$. Applying Theorem \ref{3.1}, let us define
$F(u):=w\in \tilde{S}^{2,(q/2,p/2)}([0,t]\times \mathbb{G})$ to be a unique solution to the PDE \eqref{34}. For $u_1, u_2\in \tilde{S}^{2,(q/2,p/2)}([0,t]\times \mathbb{G})$, according to the estimate \eqref{a priori 0'}, we have
\begin{multline*}
\norm{F(u_1)-F(u_2)}_{\tilde{S}^{2,(q/2,p/2)}([0,t]\times \mathbb{G})} \\
\leq C\max \{t,1\} t^{\frac{1}{2}[1-(\frac{2}{q}+\frac{Q}{p})]} \norm{b}_{L^q([0,t],L^p(\mathbb{G}))}\norm{u_1-u_2}_{\tilde{S}^{2,(r,s)}([0,t]\times \mathbb{G})}.
\end{multline*}
It follows that for a small enough $t$, a map $u \rightarrow F(u)$ is a strict contraction on $\tilde{S}^{2,(q/2,p/2)}([0,t]\times \mathbb{G})$. Thus, for a sufficiently small $0<T_1\leq t$, $F:\tilde{S}^{2,(q/2,p/2)}([t-T_1,t]\times \mathbb{G}) \rightarrow \tilde{S}^{2,(q/2,p/2)}([t-T_1,t]\times \mathbb{G})$ has a unique fixed point $u$ (note that \eqref{34} is a backward PDE). For such $T_1$ and $u$, we have
\begin{align*}
\norm{u}_{\tilde{S}^{2,(q/2,p/2)}([t-T_1,t]\times \mathbb{G})} &\leq C\max \{T_1,1\} \norm{f-\sum_{i=1}^m b^iZ_iu}_{L^{q/2}([t-T_1,t],L^{p/2}(\mathbb{G}))} \\
&\begin{aligned}
\leq C&\max \{T_1,1\} (\norm{f}_{L^{q/2}([t-T_1,t],L^{p/2}(\mathbb{G}))} \\
&+CT_1^{\frac{1}{2}[1-(\frac{2}{q}+\frac{Q}{p})]}\norm{b}_{L^q([t-T_1,t],L^p(\mathbb{G}))} \norm{u}_{\tilde{S}^{2,(q/2,p/2)}([t-T_1,t]\times \mathbb{G})}) 
\end{aligned}
\end{align*}
For small enough $T_1$, we have the estimate \eqref{04} with $u$ in place of $w$ on the interval $[t-T_1,t]$. We then redefine $u(t-T_1,x)=0$, and repeat the aforementioned argument  to obtain a solution defined on the whole interval $[0,t]$ and the estimate \eqref{04}.

Step 2. Regularization processes: since $w$ is not smooth in general, the standard It\^o's formula is not applicable to a function $w$. In order to overcome this problem,  we  take a nonnegative test function $\varphi \in C^\infty_c(\mathbb{G})$,  and introduce  mollifiers $\varphi_n(x):=n^Q\varphi(D(n)x)$.  Then, define regularized functions
 \begin{align*}
w_n(t,x):=(\varphi_n*w)(t,x)=\int_{\mathbb{G}} \varphi_n(x\circ y^{-1})w(t,y)dy.
\end{align*}    If we denote $f_n$ by 
\begin{align} \label{520}
f_n:=(w_n)_t+\sum_{i=1}^m b^iZ_iw_n +\frac{1}{2}Lw_n,
\end{align}  
then by It\^o's formula, we have
\begin{align}
w_n(t,X_t)&-w_n(s,X_s) \nonumber \\
&=\int_s^t((w_n)_t+\sum_{i=1}^m b^iZ_iw_n +\frac{1}{2}Lw_n)(r,X_r) dr+\int_s^t\sum_{i=1}^m Z_iw_n(r,X_r) dB^i_r \nonumber \\
&=\int_s^t f_n(r,X_r) dr+\int_s^t  \sum_{i=1}^m Z_iw_n(r,X_r) dB^i_r. \label{00}
\end{align}
Note that using \eqref{04} and Theorem \ref{2.13}, one can deduce that $Z_iw\in L^q([0,t],L^p(\mathbb{G}))$ for each $1\leq i\leq m$. Thus, if we denote $p'$ by a conjugate exponent of $p$, then for each $n$, 
\begin{align} \label{521}
\norm{Z_iw_n}^q_{L^q([0,t],L^\infty(\mathbb{G}))} &= \int_0^t \norm{\varphi_n * Z_iw}_{L^\infty(\mathbb{G})}^q dr \nonumber \\
&\leq \int_0^t \norm{\varphi_n}_{L^{p'}(\mathbb{G})}^q \norm{Z_iw}_{L^p(\mathbb{G})}^q dr \nonumber \\
&< \norm{\varphi_n}_{L^{p'}(\mathbb{G})}^q \norm{Z_iw}^q_{L^q([0,t],L^p(\mathbb{G}))}<\infty
\end{align} 
(note that  $\norm{Z_iw_n}_{L^q([0,t],L^\infty(\mathbb{G}))}$ may not be uniformly bounded in $n$). This implies that for each $n$, a stochastic process $r\rightarrow Z_iw_n(r,X_r)$ is square-integrable on $[0,t]$ since $q>2$  (see the condition \eqref{pq}) and
\begin{align*}
\E \Big [\int_0^t |Z_iw_n(r,X_r)|^2dr\Big] \leq \int_0^t \norm{Z_iw_n(r,\cdot)}_{L^\infty(\mathbb{G})}^2 dr 
 =  \norm{Z_iw_n}^2_{L^2([0,t],L^\infty(\mathbb{G}))} < \infty.
\end{align*}
Therefore, one can deduce that
\begin{align*}
\E\Big[\int_s^t  \sum_{i=1}^m Z_iw_n(r,X_r) dB^i_r \Big \vert \mathcal{F}_s\Big]=0.
\end{align*} 
Using this and taking a conditional expectation with respect to $\mathcal{F}_s$ in \eqref{00}, we obtain
\begin{align}
\E \Big[\int_s^t f_n(r,X_r)&dr\Big \vert \mathcal{F}_s\Big]=\E [w_n(t,X_t)-w_n(s,X_s) | \mathcal{F}_s] \nonumber \\
&\leq 2\sup_{r\in [s,t]}\norm{w_n(r,\cdot)}_{L^\infty(\mathbb{G})} \nonumber \\
&\leq C (t-s)^{1-(\frac{2}{q}+\frac{Q}{p})}(\norm{w_n}_{S^{2,(q/2,p/2)}([0,t]\times \mathbb{G})}+\norm{(w_n)_t}_{L^{q/2}([0,t],L^{p/2}(\mathbb{G}))}) \nonumber\\
&\leq C (t-s)^{1-(\frac{2}{q}+\frac{Q}{p})}(\norm{w}_{S^{2,(q/2,p/2)}([0,t]\times \mathbb{G})}+\norm{w_t}_{L^{q/2}([0,t],L^{p/2}(\mathbb{G}))}) \nonumber \\
&\leq C(t-s)^{1-(\frac{2}{q}+\frac{Q}{p})}\norm{f}_{L^{q/2}([0,t], L^{p/2}(\mathbb{G}))}. \label{01}
\end{align}
Here, we used  Theorem \ref{2.13} in the third line, convolution inequality in the fourth line, and \eqref{04} in the last line (note that $w_n(t,x)=0$). 

Now, we establish the commutator estimate.
For a.e. $r\in [s,t]$ and any $x\in \mathbb{G}$, we have that for some $0<\beta<1$,
\begin{align} \label{beta}
|\varphi_n*(b^i Z_iw) -b^i Z_i&(\varphi_n*w)|(r,x) \nonumber \\
&=\Big \vert \int_{\mathbb{G}} (b^i(r,y^{-1} \circ x)-b^i(r,x))Z_iw(r,y^{-1} \circ x)\varphi_n(y)dy \Big \vert \nonumber \\
&\leq C\int_\mathbb{G} \norm{b^i(r,\cdot)}_{S^{r-1,p}(\mathbb{G})} \norm{y}^\beta   |Z_iw(r, y^{-1} \circ x)||\varphi_n(y)| dy\nonumber \\
&\leq C\norm{b^i(r,\cdot)}_{S^{r-1,p}(\mathbb{G})}\norm{Z_iw(r,\cdot)}_{L^{\infty}(\mathbb{G})} \norm{\norm{y}^\beta \varphi_n(y)}_{L^{1}(\mathbb{G})}.
\end{align}
Here, we used  Sobolev embedding Theorem \ref{2.8} and Remark \ref{diff} in the third line (we have $(r-1)p>Q$: see Remark \ref{remark1.3}), and H\"older's inequality in the last line.  

Therefore, integrating \eqref{beta} in time and then applying H\"older's inequality, 
\begin{align} \label{5202}
&\norm{\varphi_n*(b^i Z_iw) -b^i Z_i(\varphi_n *w)}_{L^1([s,t],L^\infty(\mathbb{G}))} \nonumber \\
&\leq C\norm{\norm{y}^\beta \varphi_n(y)}_{L^{1}(\mathbb{G})}\norm{b^i}_{L^q([s,t],S^{r-1,p}(\mathbb{G}))}\norm{Z_iw}_{L^{q'}([s,t],L^{\infty}(\mathbb{G}))}
\end{align}
($q'$ is the conjugate exponent of $q$). Note that $q'<2$ since $q>2$ (see the condition \eqref{pq}). This implies that
\begin{align*}
\frac{2}{q/2}+\frac{Q}{p/2} < 2 < 1+\frac{2}{q'}.
\end{align*}
Thus, since $w\in \tilde{S}^{2,(q/2,p/2)}([0,t]\times \mathbb{G})$ and according to Theorem  \ref{2.13}, 
\begin{align}\label{5200}
\norm{Z_iw}_{L^{q'}([s,t],L^\infty(\mathbb{G}))}<\infty.
\end{align}
Also, it is obvious that
\begin{align*} 
\norm{\norm{y}^\beta \varphi_n(y)}_{L^{1}(\mathbb{G})}&=n^Q \int_\mathbb{G} \norm{y}^\beta  \varphi(D(n)y)) dy =n^{-\beta }\int_\mathbb{G} \norm{z}^{\beta } \varphi (z) dz,
\end{align*}
where the last identity is obtained by the change of variable $D(n)y=z$. Since $\varphi\in C^\infty_c(\mathbb{G})$,
\begin{align*}
\int_\mathbb{G} \norm{z}^{\beta } \varphi (z) dz<\infty,
\end{align*}
which implies that
\begin{align} \label{5201}
\lim_{n\rightarrow \infty}\norm{\norm{y}^\beta \varphi_n(y)}_{L^{1}(\mathbb{G})} = 0.
\end{align}
Therefore, using  \eqref{condition2}, \eqref{5202}, \eqref{5200}, and \eqref{5201}, we have
\be  \label{02}
\lim_{n\rightarrow \infty} \norm{\varphi_n *(b^i Z_iw) -b^i Z_i(\varphi_n *w)}_{L^1([s,t],L^\infty(\mathbb{G}))} =0.
\ee

Step 3. Proof of the estimate \eqref{itotanaka}:  since $f(r,\cdot)$ is continuous for $r$-a.e.,  $(\varphi_n * f)(r,x) \rightarrow f(r,x)$ everywhere in $x\in \mathbb{G}$ for $r$-a.e. This implies that $r$-a.e., $(\varphi_n * f)(r,X_r)\rightarrow f(r,X_r)$ for any realization $\omega\in \Omega$. Since we assumed that $f$ is non-negative  and $\varphi \geq 0$, it follows that $\varphi_n * f \geq 0$. Thus, according to the Fatou's lemma, for any realization $\omega\in \Omega$,
\begin{align}\label{525}
\int_s^t f(r,X_r)dr \leq \liminf_{n\rightarrow \infty} \int_s^t (\varphi_n*f)(r,X_r)dr.
\end{align}
Applying Fatou's lemma for the conditional expectation, 
\begin{align} \label{526}
\E\Big[\liminf_{n\rightarrow \infty}  \int_s^t (\varphi_n*f)(r,X_r)dr \Big \vert \mathcal{F}_s\Big] \leq  \liminf_{n\rightarrow \infty} \E\Big[ \int_s^t (\varphi_n*f)(r,X_r)dr \Big \vert \mathcal{F}_s\Big].
\end{align}
From \eqref{525} and \eqref{526}, we have
\begin{align} \label{527}
\E \Big [\int_s^t f(r,X_r)dr \Big \vert \mathcal{F}_s\Big] \leq \liminf_{n\rightarrow \infty} \E\Big[ \int_s^t (\varphi_n*f)(r,X_r)dr \Big \vert \mathcal{F}_s\Big].
\end{align}
On the other hand, it is easy to check that $\varphi_n * f$ can be written as
\begin{align*}
\varphi_n *f=f_n+\sum_{i=1}^m (\varphi_n*(b^i Z_iw) -b^i Z_i(\varphi_n*w)).
\end{align*}
Therefore, using this, \eqref{01}, \eqref{02}, and \eqref{527},
we obtain
\begin{align*}
\E &\Big[\int_s^t f(r,X_r) dr \Big \vert \mathcal{F}_s\Big] \leq \liminf_{n\rightarrow \infty} \E \Big[\int_s^t (\varphi_n*f)(r,X_r) dr \Big \vert \mathcal{F}_s\Big] \\
&\leq  \limsup_{n\rightarrow \infty} \E \Big[\int_s^t f_n(r,X_r) dr \Big \vert \mathcal{F}_s\Big] +\sum_{i=1}^m \E \Big[\int_s^t (\varphi_n*(b^i Z_iw) -b^i Z_i(\varphi_n*w))(r,X_r) dr \Big \vert \mathcal{F}_s\Big] \\
&\leq C(t-s)^{1-(\frac{2}{q}+\frac{Q}{p})}\norm{f}_{L^{q/2}([0,t], L^{p/2}(\mathbb{G}))}.
\end{align*}
This concludes the proof.
\end{proof}

\begin{remark}
Let us denote $X_t$ by a solution to SDE \eqref{SDE}. It is a priori not clear whether or not the integral
$\int_0^t f(s,X_s)ds$
depends on the version of $f$. In other words, it is not obvious whether or not
\begin{align*}
\int_0^t f(s,X_s)ds= \int_0^t g(s,X_s)ds
\end{align*}
holds when $f=g$ a.e. In  Proposition  \ref{4.0}, we proved the estimate \eqref{itotanaka} for continuous functions $f$ in order that $(\varphi_n * f) (r,X_r)$ converges to $f(r,X_r)$ for any realization, which enables us to apply the Fatou's lemma in Step 3 of the proof. Note that in general $(\varphi_n * f)(r,\cdot)$ converges to $f(r,\cdot)$ only at the Lebesgue point of $f(r,\cdot)$, and it is not  a priori clear whether or not $(\varphi_n * f) (r,X_r)$ converges to $f(r,X_r)$ almost surely.
\end{remark}

Since  Theorem \ref{main} is a local statement, we introduce the following notion of a solution, which is useful for our purpose:

\begin{definition}
Suppose that $\tau$ is a $\mathcal{F}_t$-stopping time.  $X_t$ is called a \it{$\tau$-solution} to SDE
\begin{align*}
dX_t=b(s,X_s)ds+\sigma(s,X_s)dB_s,\quad 0\leq t\leq T
\end{align*}
if $w$-almost surely,
\begin{align*}
X_t-X_0=\int_0^{t \wedge \tau} b(s,X_s)ds+\int_0^{t \wedge \tau} \sigma(s,X_s)dB_s
\end{align*}
holds for all $0\leq t\leq T$. 
\end{definition}
Note that if $X_t$ is a solution to SDE \eqref{SDE} and $\tau$ is any $\mathcal{F}_t$-stopping time, then $Y_t := X_{t \wedge \tau}$ is a $\tau$-solution to SDE \eqref{SDE}. The notion of $\tau$-solution is useful when we consider a stochastic process before the time at which the process exits a certain region.

\begin{remark} \label{remark 5.3}
For $\mathcal{F}_t$-stopping time $\tau$, let us denote  $\mathcal{G}_t:=\mathcal{F}_{t\wedge \tau}$. Assume that $b$ satisfies the conditions \eqref{condition1}, \eqref{condition2} for the exponents $p,q$ satisfying \eqref{pq}, and $X_t$ is $\tau$-solution to SDE \eqref{SDE}. Following the proof of Proposition \ref{4.0}, one can conclude that for any $f\in L^{q/2}([0,t], L^{p/2}(\mathbb{G}))$ such that $f(r,\cdot)$ is continuous for a.e. $r\in [0,t]$, 
\be \label{itotanaka1}
\E \Big[\int_{s\wedge \tau}^{t\wedge \tau} f(r,X_r) dr \Big \vert \mathcal{G}_s\Big]\leq  C(t-s)^{1-(\frac{2}{q}+\frac{Q}{p})}\norm{f}_{L^{q/2}([0,t], L^{p/2}(\mathbb{G}))}.
\ee
Thus, the estimate \eqref{itotanaka} is a special case of \eqref{itotanaka1} with $\tau=\infty$. The Krylov-type estimate \eqref{itotanaka1} is useful to prove the strong uniqueness of SDE \eqref{SDE} in Section \ref{section 6}.

Similarly, one can also  prove that  for any $f\in S^{r-1,(q,p)}([0,t]\times \mathbb{G})$ such that $f(r,\cdot)$ is continuous for a.e. $r\in [0,t]$,
\begin{align}\label{itotakana2}
\E \Big[\int_{s\wedge \tau}^{t\wedge \tau} f(r,X_r) dr \Big \vert  \mathcal{G}_s\Big]&\leq  C(t-s)^{1-(\frac{2}{q}+\frac{Q}{p})}\norm{f}_{L^q([0,T],L^p(\mathbb{G}))} \nonumber \\
&\leq C(t-s)^{1-(\frac{2}{q}+\frac{Q}{p})}\norm{f}_{S^{r-1,(q,p)}([0,T]\times \mathbb{G})}.
\end{align}  
\end{remark}

Using the Krylov-type estimate \eqref{itotakana2}, one can derive the It\^o's formula for the mixed-norm parabolic Sobolev spaces $\tilde{S}^{r+1,(q,p)}([0,T]\times \mathbb{G})$.
\begin{theorem} \label{A.1}
Suppose that assumptions in Theorem \ref{main} are satisfied, and $X_t$ is a $\tau$-solution to  the SDE:
\begin{align*}
dX_t=b(t,X_t)dt+\sum_{i=1}^m Z_i(t,X_t) \circ dB^i_t, \quad 0\leq t\leq T.
\end{align*}
Then, for any $f\in \tilde{S}^{r+1,(q,p)}([0,T]\times \mathbb{G})$ satisfying
\begin{gather*}
f \quad \text{contiuous in} \ (t,x), \\
 (f_t+\sum_{i=1}^m b^i Z_if+\frac{1}{2}\sum_{i=1}^m Z_i^2f)(t,\cdot), \  Z_1f(t,\cdot),\cdots,Z_mf(t,\cdot) \quad \text{continuous in x for t-a.e.} ,
\end{gather*}
a process $f(t,X_t)$ is a $\tau$-solution to 
\begin{align*}
df(t,X_t)=(f_t+\sum_{i=1}^m b^i Z_if+\frac{1}{2}\sum_{i=1}^m Z_i^2f)(t,X_t)dt+ \sum_{i=1}^m Z_i  f(t,X_t) dB^i_t, \quad 0\leq t\leq T.
\end{align*}

\end{theorem}
\begin{proof}
Since $t\rightarrow f(t,X_t)$ is continuous, it suffices to check that for each $t$,
\begin{multline} \label{stra}
f(t\wedge \tau,X_{t\wedge \tau}) \\
=\int_0^{t\wedge \tau}(f_t+\sum_i b^i Z_if+\frac{1}{2}\sum_i Z_i^2f)(s,X_s) ds+\int_0^{t\wedge \tau} \sum_i Z_i f(s,X_s) dB^i_s
\end{multline}
holds almost surely. Let us approximate $f$ by smooth functions $f_n$ in $\tilde{S}^{r+1,(q,p)}([0,T]\times \mathbb{G})$ norm. More precisely, for a mollifier $\varphi_n(x):=n^Q\varphi(D(n)x)$ with $\varphi \in C^\infty_c(\mathbb{G})$, let us define $f_n:=\varphi_n * f$.  Then,  
\begin{align*}
(f_n)_t+\sum_i b^i Z_if_n+\frac{1}{2}\sum_i Z_i^2f_n \rightarrow f_t+\sum_i b^i Z_if+\frac{1}{2}\sum_i Z_i^2f \quad \text{in} \ S^{r-1,(q,p)}.
\end{align*}
Since for $t$-a.e., both $\Big[(f_n)_t+\sum_i b^i Z_if_n+\frac{1}{2}\sum_i Z_i^2f_n\Big](t,\cdot)$ and $ (f_t+\sum_i b^i Z_if+\frac{1}{2}\sum_i Z_i^2f)(t,\cdot)$ are continuous, using the estimate \eqref{itotakana2}, one can conclude that
\begin{align}
\lim_{n\rightarrow \infty} \E \Big \vert\int_0^{t\wedge \tau} ((f_n)_t+&\sum_i b^i Z_if_n+\frac{1}{2}\sum_i Z_i^2f_n)(s,X_s)ds \nonumber \\ &-\int_0^{t\wedge \tau} (f_t+\sum_i b^i Z_if+\frac{1}{2}\sum_i Z_i^2f)(s,X_s)ds\Big \vert = 0. \label{a11}
\end{align}
Also, since for $t$-a.e., both $Z_if_n(t,\cdot)$ and $Z_if(t,\cdot)$ are continuous, using  the It\^o's isometry and \eqref{itotakana2}, we have
\begin{align}
\lim_{n\rightarrow \infty} \E \Big \vert\int_0^{t\wedge \tau} \sum_i &  Z_if_n(s,X_s) dB^i_s -\int_0^{t\wedge \tau} \sum_i Z_if(s,X_s) dB^i_s\Big \vert^2 \nonumber \\= &\lim_{n\rightarrow \infty} \E \int_0^{t\wedge \tau} \Big[\sum_i Z_i(f_n-f)\Big]^2(s,X_s) ds \nonumber \\  
&\leq C\lim_{n\rightarrow \infty} \norm{\Big[\sum_i Z_i(f_n-f)\Big]^2}_{S^{r-1,(q,p)}} \nonumber \\
&\leq C\lim_{n\rightarrow \infty} \norm{f_n-f}^2_{S^{r,(2q,2p)}} \nonumber \\
&\leq C\lim_{n\rightarrow \infty}\norm{f_n-f}^2_{\tilde{S}^{r+1,(q,p)}}= 0. \label{a12}
\end{align}
Note that parabolic Sobolev embedding Theorem \ref{2.13} is applicable in the last line since $\frac{2}{q}+\frac{Q}{p} < 1+ \frac{2}{2q}+\frac{Q}{2p}$.
Furthermore, according to Theorem \ref{2.13} again,
\begin{align*}
\norm{f_n-f}_{L^\infty} \leq  \norm{f_n-f}_{S^{r,(\infty,\infty)}} \leq C\norm{f_n-f}_{\tilde{S}^{r+1,(q,p)}}.
\end{align*}
Therefore,
\begin{align}
\lim_{n\rightarrow \infty} |f_n(t\wedge \tau,X_{t\wedge \tau})-f(t\wedge \tau,X_{t\wedge \tau})|\leq \lim_{n\rightarrow \infty}  \norm{f_n-f}_{L^\infty}=0. \label{a13}
\end{align} Since $f_n$'s are smooth, the classical It\^o's formula yields that
\begin{multline}
f_n(t\wedge \tau,X_{t\wedge \tau}) \\
=\int_0^{t\wedge \tau}((f_n)_t+\sum_i b^i Z_if_n+\frac{1}{2}\sum_i Z_i^2f_n)(s,X_s) ds+\int_0^{t\wedge \tau} \sum_i Z_i f_n(s,X_s) dB^i_s.
\end{multline}
Therefore,  sending $n\rightarrow \infty$ along the appropriate subsequence using \eqref{a11}, \eqref{a12}, and \eqref{a13}, we  obtain \eqref{stra}.
\end{proof}

\section{Proof of the main theorem \ref{main}} \label{section 6}
In this section, based on the  results established in Section \ref{section 4} and \ref{section 5}, we prove the main result Theorem \ref{main}. We identify the space $\mathbb{G}$ with the Euclidean space $\R^N$, and then we do a stochastic calculus. Throughout this section, we add a time parameter $t$ to the time independent vector fields $Z_i$'s, i.e. $Z_i(t,x)=Z_i(x)$ for $0\leq t\leq T$.

\subsection{Conjugated SDE} \label{section 6.1}
In this section, we derive an auxiliary SDE transformed by the original SDE \eqref{SDE}, which is called a \it{conjugated SDE}. The advantage of this new SDE over the original SDE is that it possesses a more regular drift coefficient. This idea goes back to the Zvonkin's work \cite{zvo}, and  has been successfully used  to prove the well-posedness of SDEs with the additive noise (see for example \cite{ff3,zhang1}). In the next proposition, as in \cite{ff3,zhang1}, we obtain an auxiliary SDE  using a function $u$. Recall that functions $u$, $\Phi$, and the open set $\Omega$ are defined in Remark \ref{remark36}.

\begin{proposition} \label{4.1}
For $0\leq t\leq T$ and $x\in \Phi(t,\Omega)$, let us define vector fields $\tilde{b}$ and $\tilde{\sigma}_i$ ($1\leq i\leq m$) via
\begin{align*}
\tilde{b}(t,x)=[\lambda u+ \frac{1}{2} \sum_{i=1}^m Z_i'Z_i](t,\Phi^{-1}(t,x)), \quad
\tilde{\sigma}_i(t,x)=(
Z_i+Z_iu)(t,\Phi^{-1}(t,x))
\end{align*}
($Z_i':\R^N \rightarrow \R^{N\times N}$ is a standard derivative of the map $Z_i:\R^N \rightarrow \R^N$, and $Z_i'Z_i$ is interpreted as a product of $N\times N$ matrix $Z_i'$ and a vector $Z_i\in \R^N$).   Suppose that $X_t$ is a $\tau$-solution to SDE  \eqref{SDE} for a $\mathcal{F}_t$-stopping time $\tau$ such that $X_t\in \Omega$ for $0\leq t\leq T$. Then, $Y_t=\Phi(t,X_t)$ is a $\tau$-solution to the following SDE:
\be \label{conjugated'}
\begin{cases}
dY_t=\tilde{b}(t,Y_t)dt+\sum_{i=1}^m \tilde{\sigma}_i(t,Y_t)dB^i_t,\\
Y_0=\Phi(0,x_0).
\end{cases}
\ee
\end{proposition}
\begin{proof}
In order to alleviate notations, we omit the summation symbol $\sum_i$. Recall that $u(t,x)$ is continuous, $(\partial_t u+ b^i Z_iu+\frac{1}{2}Lu)(s,\cdot)$, $Z_iu(s,\cdot)$ are continuous for $s$-a.e. (see Remark \ref{remark36}), and $u\in \tilde{S}^{r+1,(q,p)}([0,T]\times \mathbb{G})$. Therefore, using the It\^o's formula for non-smooth functions  Theorem \ref{A.1}, we have
\begin{align*}
&u(t,X_t)=u(0,X_0)+\int_0^{t \wedge \tau} (\partial_t u+ b^i Z_iu+\frac{1}{2}Lu)(s,X_s)ds+\int_0^{t \wedge \tau} Z_iu(s,X_s)dB_s^i \\
&= u(0,X_0)-\int_0^{t \wedge \tau} (b-\lambda u)(s,X_s)ds+\int_0^{t \wedge \tau} Z_iu(s,X_s)dB_s^i \\
&=u(0,X_0)-X_t+X_0+\int_0^{t \wedge \tau} \lambda u(s,X_s)ds+ \int_0^{t \wedge \tau}  Z_i(X_s) \circ dB^i_s+\int_0^{t \wedge \tau} Z_iu(s,X_s)dB^i_s \\
&=u(0,X_0)-X_t+X_0+\int_0^{t \wedge \tau}[ \lambda u+ \frac{1}{2} Z_i'Z_i](s,X_s)ds+ \int_0^{t \wedge \tau} (Z_i+ Z_iu)(s,X_s)dB^i_s.
\end{align*}
Since $X_t\in \Omega$ for $0\leq t\leq T$, $Y_t\in \Phi(t,\Omega)$. Therefore,
\begin{align*}
Y_t&-Y_0=\Phi(t,X_t)-\Phi(t,X_0) \\
&=\int_0^{t \wedge \tau}[ \lambda u+ \frac{1}{2}Z_i'Z_i](s,X_s)ds+ \int_0^{t \wedge \tau} (Z_i+ Z_iu)(s,X_s)dB^i_s \\
&=\int_0^{t \wedge \tau}[ \lambda u+ \frac{1}{2}Z_i'Z_i](s,\Phi^{-1}(s,Y_s))ds+ \int_0^{t \wedge \tau}  (Z_i+ Z_iu)(s,\Phi^{-1}(s,Y_s))dB^i_s \\
&=\int_0^{t \wedge \tau} \tilde{b}(s,Y_s)ds+\int_0^{t \wedge \tau}  \tilde{\sigma}_i(s,Y_s)dB_t^i.
\end{align*}
\end{proof}

\subsection{Strong uniqueness} \label{section 6.2}
Using the conjugated SDE \eqref{conjugated'},  one can prove that a strong solution to SDE \eqref{SDE} is unique:
\begin{theorem} \label{4.3}
 Suppose that $X^1_t$, $X^2_t$ are $\tau$-solutions to \eqref{SDE} for a $\mathcal{F}_t$-stopping time $\tau$ such that $X^1_t, X^2_t\in \Omega'$ for $0\leq t\leq T$. Then, $X^1_t=X^2_t$ almost surely.
\end{theorem}
\begin{proof}
Let us define $Y^k_t=\Phi(t,X^k_t)$ for $k=1,2$. Then, according to Proposition \ref{4.1}, $Y_t^k$ is a $\tau$-solution to SDE \eqref{conjugated'}, and $Y_t^k \in \Phi(t,\Omega)$ for each $t$. Thus,
\be  \label{620}
Y^1_t-Y^2_t=\int_0^{t \wedge \tau}[\tilde{b}(s,Y^1_s)-\tilde{b}(s,Y^2_s)]ds+ \sum_{i=1}^m \int_0^{t \wedge \tau}[\tilde{\sigma}_i(s,Y^1_s)-\tilde{\sigma}_i(s,Y^2_s)]dB^i_s.
\ee
Let us first check that $\tilde{b}(t,\cdot)$ is Lipschitz continuous on $ \Phi(t,\Omega)$ uniformly in $t$. Note that $\norm{\nabla u}_{L^\infty([0,T]\times \Omega)} \leq \frac{1}{2}$ (see Step 3 of the proof of Proposition \ref{3.5}) and a map $x\rightarrow Z_i'Z_iu(x)$ is smooth on $\R^N$.  Thus, applying a chain rule, we obtain the uniform Lipschitz continuity of $\tilde{b}(t,\cdot)$  since   $\norm{\nabla \Phi^{-1}(t,\cdot)}_{L^\infty(\Phi(t,\Omega))}$ is uniformly bounded in $t$ (see Proposition \ref{3.5}).

 Therefore, using this fact and applying the It\^o's formula to \eqref{620}, for any $a>2$,
\begin{align}
&d|Y^1_t-Y^2_t|^a \nonumber \\
&\leq (L|Y^1_s-Y^2_s|^a+\frac{a(a-1)}{2} |\tilde{\sigma}(s,Y^1_s)-\tilde{\sigma}(s,Y^2_s)|^2|Y^1_s-Y^2_s|^{a-2}) \1_{[0,\tau]}ds +  W_s \1_{[0,\tau]} dB_s \label{423}
\end{align}
for some constant $L>0$ and the process $W_s$ satisfying
\be \label{425}
|W_s| \leq C|Y^1_s-Y^2_s|^{a-1} |\tilde{\sigma}(s,Y^1_s)-\tilde{\sigma}(s,Y^2_s)|.
\ee 
Here, $\tilde{\sigma}$ denotes a $N\times m$ matrix whose columns consist of $\tilde{\sigma}_i$'s. In order to deal with the right hand side of \eqref{423}, we need the following lemma, motivated by \cite[Lemma 4.4]{ff3} and \cite[Lemma 5.4]{KR}. 
\begin{lemma} \label{4.4}
There exists a continuous and $\mathcal{F}_t$-adapted process $A_t$ satisfying
\be \label{4.4 1}
\frac{a(a-1)}{2}\int^t_0|\tilde{\sigma}(s,Y^1_s)-\tilde{\sigma}(s,Y^2_s)|^2\1_{[0,\tau]}ds=\int^t_0|Y^1_s-Y^2_s|^2dA_s
\ee
and 
\be \label{exp}
\E e^{cA_s}<\infty
\ee
for any $c>0$.
\end{lemma}
\begin{proof}
Let us define a process $A_t$ by
\begin{align*}
A_t := \frac{a(a-1)}{2}\int_0^t \1_{Y^1_s \neq Y^2_s}\frac{|\tilde{\sigma}(s,Y^1_s)-\tilde{\sigma}(s,Y^2_s)|^2}{|Y^1_s-Y^2_s|^2}\1_{[0,\tau]}ds.
\end{align*} 
Then, it is obvious that $A_t$ satisfies \eqref{4.4 1}, and it suffices to prove the estimate \eqref{exp}. Note that since $Y_t^k\in \Phi(t,\Omega)$, using the property (iii) in Proposition \ref{3.5},
\begin{align*}
|X^1_t-X^2_t| = |\Phi^{-1}(t,Y^1_t)-\Phi^{-1}(t,Y^2_t)| \leq 2|Y^1_t-Y^2_t|.
\end{align*}
Using this, we have
\begin{align} \label{4444}
A_t 
&\leq C\sum_i \int_0^t \1_{Y^1_s \neq Y^2_s}\frac{|(
Z_i+Z_iu)(s,X^1_s)-(Z_i+Z_iu)(s,X^2_s)|^2}{|Y^1_s-Y^2_s|^2}\1_{[0,\tau]}ds \nonumber \\
&\leq C\sum_i \int_0^t \1_{X^1_s \neq X^2_s}\frac{|(
Z_i+Z_iu)(s,X^1_s)-(Z_i+Z_iu)(s,X^2_s)|^2}{|X^1_s-X^2_s|^2}\1_{[0,\tau]}ds.
\end{align}
Here, we used the Lipschitz continuity of $\Phi^{-1}(t,\cdot)$ (see Proposition \ref{3.5}). For mollifiers $\rho_n(x):=n^N\rho(nx)$ with $\rho \in  C^\infty_c(\R^N)$, let us first prove that for any $c\in \R$,
 \begin{multline} \label{422}
\limsup_n \E \exp \Big[ c \int_0^{t\wedge \tau} \1_{X^1_s \neq X^2_s}\\
\cdot \frac{|[(Z_i+Z_iu)*\rho_n](s,X^1_s)-[(Z_i+Z_iu)*\rho_n](s,X^2_s)|^2}{|X^1_s-X^2_s|^2}ds\Big]<\infty.
\end{multline}
Here, $*$ denotes the standard convolution operator on the Euclidean spaces: $(f*g)(x) = \int_{\R^N} f(x-y)g(y)dy$.
Let us choose a cutoff function $\phi\in C_c^\infty(\R^N)$ such that $\phi=1$ on $\Omega$, and denote $K:=\text{supp}(\phi)$. Also, denote $\mathcal{M}$ by a Hardy-Littlewood maximal operator with respect to the Euclidean distance and Lebesgue measure on $\R^N$. Since $X^k_t\in \Omega$, we have
\begin{align} \label{6.00}
&\frac{|[(Z_i+Z_iu)*\rho_n](s,X^1_s)-[(Z_i+Z_iu)*\rho_n](s,X^2_s)|^2}{|X^1_s-X^2_s|^2}\nonumber \\
&=\frac{|\phi\cdot \{(Z_i+Z_iu)*\rho_n\}(s,X^1_s)-\phi\cdot \{(Z_i+Z_iu)*\rho_n\}(s,X^2_s)|^2}{|X^1_s-X^2_s|^2}\nonumber \\
&\leq C(|\mathcal{M}\nabla[\phi\cdot \{(Z_i+Z_iu)*\rho_n\}]|^2(s,X^1_s)+|\mathcal{M}\nabla[\phi \cdot \{(Z_i+Z_iu)*\rho_n\}]|^2(s,X^2_s)).
\end{align}
Here, we used the fact that for some constant $C=C(N)$, the inequality 
\be \label{lemma}
|f(x)-f(y)| \leq C|x-y|(\mathcal{M}\nabla f(x)+\mathcal{M}\nabla f(y))
\ee
holds for any $f\in C^\infty(\R^N)$. 
On the other hand, using the fact that $\phi$ has compact support $K$ and $u\in S^{r+1,(q,p)}$,  for any $n$,
\begin{align*}
&\norm{\nabla[\phi\cdot \{(Z_i+Z_iu)*\rho_n\}]}_{L^q([0,T], L^p(\R^N))} \\
&\leq \norm{\phi \cdot  \{[\nabla(Z_i+Z_iu)] *\rho_n\}}_{L^q([0,T], L^p(\R^N))} + \norm{\nabla \phi \cdot \{(Z_i+Z_iu)*\rho_n\}}_{L^q([0,T], L^p(\R^N))} \\
&\leq C\big(\norm{\nabla(Z_i+Z_iu)}_{L^q([0,T], L^p(K))} + \norm{Z_i+Z_iu}_{L^q([0,T], L^p(K))}\big) \\
&\leq C(1+\norm{u}_{S^{r+1, (q,p)}([0,T]\times \R^N)}).
\end{align*} 
Here, we used the convolution inequality in the second line. Also, in the third line, we used 
\begin{align*}
\norm{\nabla (Z_iu)}_{L^q([0,T],L^p(K))} \leq C(K)\norm{u}_{S^{r+1, (q,p)}([0,T]\times \R^N)}
\end{align*}
(recall that each standard vector field on $\R^N$ can be written as a linear combination of commutators of $Z_i$'s with order $\leq r$).

Therefore, we obtain
\begin{align} \label{4.20}
\limsup_n &\norm{ |\mathcal{M}\nabla[\phi\cdot \{(Z_i+Z_iu)*\rho_n\}]|^2}_{L^{q/2}([0,T], L^{p/2}(\R^N))} \nonumber \\ 
&\leq C\limsup_n \norm{ |\nabla[\phi\cdot \{(Z_i+Z_iu)*\varphi_n\}]|^2}_{L^{q/2}([0,T], L^{p/2}(\R^N))} < \infty
\end{align}
since the maximal operator $\mathcal{M}$ is bounded in $L^p(\R^N)$.
Since $\mathcal{M}\nabla[\phi \cdot \{ (Z_k+Z_ku)*\rho_n\}](s,\cdot)$ is continuous  $s$-a.e, using \eqref{itotanaka1}, one can conclude that
\begin{align} \label{421}
&\limsup_n\E \Big[\int_{s\wedge \tau}^{t\wedge \tau} |\mathcal{M}\nabla[\phi\cdot  \{(Z_i+Z_iu)*\rho_n\}]|^2(r,X^k_r) dr\Big \vert\mathcal{G}_s\Big] \nonumber \\
&\leq  C(t-s)^{1-(\frac{2}{q}+\frac{Q}{p})}\limsup_n\norm{|\mathcal{M}\nabla[\phi\cdot \{ (Z_i+Z_iu)*\rho_n\}]|^2}_{L^{q/2}([0,t], L^{p/2}(\R^N))}. 
\end{align}
Therefore, using \eqref{6.00}, \eqref{4.20}, \eqref{421}, and Remark \ref{a.5}, we obtain the estimate \eqref{422}. 

Finally, let us check that \eqref{422} implies \eqref{exp}. Since $(Z_i+Z_iu)(s,\cdot)$ is continuous for $s$-a.e, $[(Z_i+Z_iu) * \rho_n](s,\cdot)$ converges  to $(Z_i+Z_iu)(s,\cdot)$ pointwisely in $x\in \R^N$ for $s$-a.e. Thus, using \eqref{4444}, \eqref{422}, and the Fatou's lemma, we obtain \eqref{exp}.
\end{proof}
Let us go back to the proof of  Theorem \ref{4.3}. Applying Lemma \ref{4.4} to \eqref{423}, we have
\begin{align}
e^{-A_t}|Y^1_t-Y^2_t|^a&=\int_0^t -e^{-A_s}|Y^1_s-Y^2_s|^a dA_s+\int_0^t e^{-A_s}d|Y^1_s-Y^2_s|^a \nonumber \\
&\leq \int_0^t Le^{-A_s}|Y^1_s-Y^2_s|^a\1_{[0,\tau]}ds+\int_0^t e^{-A_s} W_s\1_{[0,\tau]} dM_s. \label{424}
\end{align}
Let us define a $\mathcal{F}_t$-stopping time $\tau_l$ by \begin{align*}
\tau_l=\inf \{0\leq t\leq T \ | \ | Y^1_t| >l \ \text{or} \ |Y^2_t|>l\},
\end{align*}
and  $\tau_l=T$ if the above set is empty. Then, by \eqref{424},
\begin{align}
e^{-A_{t\wedge \tau_l}}&|Y^1_{t  \wedge \tau_l}-Y^2_{t  \wedge \tau_l}|^a \nonumber \\
&\leq \int_0^t Le^{-A_s}|Y^1_s-Y^2_s|^a\1_{[0,\tau]}\1_{[0,\tau_l]}ds+\int_0^t e^{-A_s} W_s\1_{[0,\tau]}\1_{[0,\tau_l]} dM_s. \label{426}
\end{align}
Let us check that  that $\tilde{\sigma}$ is bounded on $\Phi(t,\Omega)$ uniformly in $t$. Since $u\in \tilde{S}^{r+1,(q,p)}$ (see Remark \ref{remark36}), according to Theorem \ref{2.13}, $Z_iu\in L^\infty([0,T]\times \R^N)$. Also, it is obvious that $Z_i(\cdot)$ is  bounded on $\Omega$. These facts  imply the uniform boundedness of $\tilde{\sigma}(t,\cdot)$ on $\Phi(t,\Omega)$.

 Thus, since  $|Y^1_s|,|Y^2_s| \leq l$ for $s\in [0,\tau_l]$, for some constant $C$,
\begin{align*}
|Y^1_s-Y^2_s|^{a-1} |\tilde{\sigma}(s,Y^1_s)-\tilde{\sigma}(s,Y^2_s)|<Cl^{a-1}
\end{align*}
 for any  $s\in [0,\tau_l]$. From this and \eqref{425}, it follows that $s\mapsto e^{-A_s} W_s\1_{[0,\tau]}\1_{[0,\tau_l]}$ is a square-integrable process. Therefore, taking the expectation in \eqref{426}, 
\begin{align*}
\E [e^{-A_{t\wedge \tau_l}}&|Y^1_{t \wedge \tau_l}-Y^2_{t \wedge \tau_l}|^a ] \leq L\int_0^t \E [e^{-A_s}|Y^1_s-Y^2_s|^a\1_{[0,\tau]}\1_{[0,\tau_l]}]ds.
\end{align*}
Sending $l\rightarrow \infty$ and applying the  Fatou's lemma,
\begin{align*}
\E [e^{-A_{t}}&|Y^1_t-Y^2_t|^a ] \leq L\int_0^t \E [e^{-A_s}|Y^1_s-Y^2_s|^a\1_{[0,\tau]}]ds.
\end{align*}
Applying the Gronwell's inequality, we obtain
\begin{align*}
\E [e^{-A_t}|Y^1_t-Y^2_t|^a] =0.
\end{align*}
Using the Holder's inequality,
\begin{align*}
\E |Y^1_t-Y^2_t|^{a/2} &\leq [\E e^{-A_t}|Y^1_t-Y^2_t|^{a} ]^{1/2}[\E e^{A_t}]^{1/2}=0
\end{align*}
since $\E e^{A_t}$ is finite (see the estimate \eqref{exp}). 
Thus, we have
\begin{align*}
\E|Y^1_t-Y^2_t|^{a/2}=0.
\end{align*} Since trajectories are continuous in time and $\Phi(t,\cdot)$ is bijective from $\Omega$ onto $\Phi(t,\Omega)$ for each $t$, proof is concluded.
\end{proof}
\subsection{Conclusion of the proof of Theorem \ref{main}} \label{section 6.3}
In this section, we finally complete the proof of the main result Theorem \ref{main}. As mentioned in Section \ref{section 2}, we show  the weak existence and strong uniqueness separately, and then apply the Yamada-Watanabe principle. Since we have already proved the uniqueness of a solution in Section \ref{section 6.2}, it suffices to derive the existence of a weak solution. Let us first recall the well-known fact about the existence of a weak solution:
\begin{theorem}\cite{var} \label{4.5}
Suppose that $b(t,\cdot)$ and $\sigma(t,\cdot)$ are continuous in $x$ and have linear growth  for each $0\leq t\leq T$. Then, SDE 
\begin{align*}
\begin{cases}
dX_t=b(t,X_t)dt+\sigma(t,X_t)dB_t, \quad 0\leq t\leq T,\\
X_0 = x_0,
\end{cases}
\end{align*}
admits a weak solution.
\end{theorem}
Since Theorem \ref{main} is a local statement, we localize coefficients of SDE  \eqref{SDE}, and then apply Theorem \ref{4.5}. More precisely,  choose a cutoff function $\psi\in C^\infty_c(\R^N)$ such that $\psi=1$ on $\Omega$, and  consider the following SDE:
\be \label{451}
\begin{cases}
dX_t=(\psi b) (t,X_t)dt+\sum_{i=1}^m (\psi Z_i)(t,X_t)\circ dB^i_t,\quad 0\leq t\leq T, \\
X_0 = x_0.
\end{cases}
\ee
\begin{corollary} \label{4.6}
There exists a weak solution to SDE \eqref{451}.
\end{corollary}
\begin{proof}
Recall that $b(t,\cdot)$ is continuous for $t$-a.e (see Remark \ref{remark1.3}) and $\psi$ is a cutoff function. Thus, according to Theorem \ref{4.5}, SDE \eqref{451} has a weak solution.
\end{proof}
Now, we are ready to conclude the proof of  the main result Theorem \ref{main}.
\begin{proof}[Proof of Theorem \ref{main}]
Let us consider the following SDE:
\be \label{11}
\begin{cases}
dX_t=b_t(X_{\cdot}) \1_{t<\tau(X_{\cdot})} dt+\sum_{i=1}^m (Z_i)_t(X_{\cdot})\1_{t<\tau(X_{\cdot})} \circ dB^i_t, \quad 0\leq t\leq T, \\
X_0 = x_0.
\end{cases}
\ee 
Here, $b_t(x_\cdot)$, $(Z_i)_t(x_\cdot)$ are $\R^N$-valued progressive functions  on the space $[0,T] \times C([0,T], \R^N)$, equipped with the canonical filtration $\mathcal{F}_t=\sigma \{x_s | s\leq t \}$,  defined by $b_t(x_{\cdot}):=b(t,x_t)$, $(Z_i)_t(x_{\cdot}):=Z_i(t,x_t)$. Also, $\mathcal{F}_t$-stopping time $\tau$ is defined by $\tau(x_{\cdot}):=\inf \{t\leq T \ | \ x_t \not\in \Omega \}$ and $\tau(x_{\cdot})=T$ if the set is empty.
Uniqueness of a strong solution to SDE \eqref{11} follows from Theorem \ref{4.3}, and the existence of a weak solution follows from Corollary \ref{4.6}. Therefore, Yamada-Watanabe principle (see Theorem \ref{a.3}) concludes that a unique strong solution exists to SDE \eqref{11}. This concludes the proof of Theorem \ref{main}. 
\end{proof}

\appendix
\section{Lemmas in the probability theory} \label{section a}
In this Appendix \ref{section a}, we review key  notions and lemmas in the probability theory frequently used throughout the paper. 
First, we recall the definition of a weak solution and a strong solution to the SDE:
\begin{definition}\cite[Chapter 18]{kal}  \label{definition2}
Consider SDE of the following form:
\be \label{def SDE}
dX_t=b_t(X_\cdot)dt+\sigma_t(X_\cdot)dB_t.
\ee 
Here, $b$ and $\sigma$ are progressive functions defined on $\R_+ \times C(\R_+, \R^d)$ equipped with the canonical filtration $\mathcal{F}_t=\sigma \{x_s | s\leq t \}$.
For a given filtered probability space $(\Omega, \mathcal{F}, \mathcal{F}_t, P)$, $\mathcal{F}_t$-Brownian motion $B$, and an $\mathcal{F}_0$-measurable random variable $\xi$, $X$ is a \it{strong solution} to SDE if it is a $\mathcal{F}_t$-adapted process with $X_0=\xi$ solving \eqref{def SDE} almost surely. For a given initial distribution $\mu$, a \it{weak solution} consists of the filtered probability space $(\Omega, \mathcal{F}, \mathcal{F}_t, P)$, $\mathcal{F}_t$-Brownian motion $B$, and a $\mathcal{F}_t$-adapted process $X$ with $P \circ X_0^{-1}=\mu$ satisfying \eqref{def SDE} almost surely. 

We say that \it{weak existence} holds for the initial distribution $\mu$ if there exists a weak solution $(\Omega, \mathcal{F}, \mathcal{F}_t, P, B, X)$ satisfying \eqref{def SDE}. \it{Strong existence} is said to  hold for the initial distribution $\mu$ if there exists a strong solution $X$ for every $(\Omega, \mathcal{F}, \mathcal{F}_t, P, B, \xi)$ satisfying $P \circ \xi^{-1} = \mu$. We say that \it{strong uniqueness} holds for the initial distribution $\mu$ provided that for any solutions $X$ and $Y$ to \eqref{def SDE} on the common filtered probability space with a given Brownian motion such that $X_0=Y_0$ a.s. with a distribution $\mu$, $X=Y$ almost surely. Finally, \it{weak uniqueness} is said to  hold for the initial distribution $\mu$ if each weak solution $X$ has the same distribution.
\end{definition}
The following theorem proved by Watanabe and Yamada \cite{YW,YW1} is crucial to prove the existence of a strong solution to SDE.
\begin{theorem}[Yamada-Watanabe Principle, \cite{YW,YW1}]\label{a.3} Consider the following SDE:
\be \label{yw}
dX_t=b(t,X_t)dt+\sigma(t,X_t)dB_t,
\ee
with a given initial condition. 
Suppose that a weak solution to \eqref{yw} exists and a strong solution to \eqref{yw} is unique. Then, the strong existence and weak uniqueness hold as well.
\end{theorem}
The following lemma is crucially used in the proof of Proposition \ref{4.3}.
\begin{lemma}[\cite{exp}] \label{a.2}
Let $X_t$ ($0\leq t\leq T$) be a nonnegative stochastic process adapted to $\mathcal{F}_t$. Assume that for any $0\leq s\leq t\leq T$,
\begin{align*}
\E\Big[\int_s^t X_r dr \Big \vert \mathcal{F}_s\Big] \leq f(s,t)
\end{align*}
holds for some deterministic function $f(s,t)$ satisfying \\
(i) $f(s_1,t_1)\leq f(s_2,t_2)$ for $[s_1,t_1] \subset [s_2,t_2]$. \\
(ii) $\lim_{h\rightarrow 0^+}\sup_{0\leq s\leq t\leq T, |t-s|\leq h} f(s,t)=\alpha \geq 0$. \\
Then, for arbitrary $c<\alpha^{-1}$ (when $\alpha=0$, $\alpha^{-1}$ is defined by $\alpha^{-1}:=\infty$),
\be \label{exp lemma}
\E \exp\Big[c\int_0^T X_rdr\Big] <\infty.
\ee 
\end{lemma}
\begin{remark} \label{a.5}
In fact, the left hand side of \eqref{exp lemma} can be controlled in terms of $f$ (see the proof of  \cite[Lemma 1.1]{exp}). Also, Lemma \ref{a.2} can be generalized as follows: assume that for a $\mathcal{F}_t$-stopping time $\tau$ and a nonnegative process $X_t$ adapted to $\mathcal{F}_t$,
\begin{align*}
\E\Big[\int_{s\wedge \tau}^{t\wedge \tau} X_r dr \Big \vert \mathcal{F}_s\Big] \leq f(s,t)
\end{align*}
holds for some deterministic function $f$ satisfying the conditions (i), (ii) in Lemma \ref{a.2}. Then, for any $c<\alpha^{-1}$,
\begin{align*}
\E \exp\Big[c\int_0^{T\wedge \tau} X_rdr\Big] <\infty.
\end{align*}
This immediately follows from Lemma \ref{a.2} due to  the identity $\int_{s\wedge \tau}^{t\wedge \tau} X_r dr=\int_s^t X_r \1_{r \leq \tau} dr$ and the fact that a process $t \mapsto X_t \1_{t\leq \tau}$ is $\mathcal{F}_t$-adapted.
\end{remark}

\section{Mixed-norm parabolic Sobolev embedding theorem}\label{section b}

In this Appendix \ref{section b}, we obtain the parabolic Sobolev embedding theorem for the spaces $S^{k,(q,p)}([0,T]\times \mathbb{G})$ ($\mathbb{G}$ is a homogeneous Carnot group). This is a key ingredient to establish the well-posedness result of the Kolmogorov PDE possessing singular coefficients (see Section \ref{section 4}).

\begin{theorem} \label{2.13}
Suppose that $u$ satisfies $u(0,x)=0$ and
\begin{align*}
u\in S^{k+2,(q,p)}([0,T]\times \mathbb{G}),\quad u_t\in S^{k,(q,p)}([0,T]\times \mathbb{G}).
\end{align*} Also, assume  that  for $l=k,k+1$,  exponents $p,q,p_1,q_1$ satisfy 
\begin{align} \label{b condition}
 1\leq p\leq p_1\leq \infty, \ 1\leq q\leq q_1\leq  \infty, \ \frac{2}{q}+\frac{Q}{p}<(k+2-l)+\frac{2}{q_1}+\frac{Q}{p_1}.
\end{align}
Then, $u\in {S^{l,(q_1,p_1)}}([0,T]\times \mathbb{G})$.  Also, if we denote $
\alpha:=\frac{1}{2}\Big[(k+2-l+\frac{2}{q_1}+\frac{Q}{p_1})-(\frac{2}{q}+\frac{Q}{p})\Big]$, then for some constant $C$ independent of $T$ and $u$,
\be \label{2131}
\norm{u}_{S^{l,(q_1,p_1)}([0,T]\times \mathbb{G})} \leq CT^{\alpha}( \norm{u}_{S^{k+2,(q,p)}([0,T]\times \mathbb{G})}+\norm{u_t}_{S^{k+2,(q,p)}([0,T]\times \mathbb{G})}).
\ee
\end{theorem}

\begin{proof}
It suffices to prove the estimate \eqref{2131} for all test functions $u$. 

Step 1. The case  $l=k$: for any indices $|I|\leq l$, let us define $w=Z_Iu$. Then, for any indices $1\leq i,j\leq m$,
\begin{align*}
w,\, w_t,\, Z_i,\, Z_iZ_jw \in L^q([0,T],L^p(\mathbb{G})).
\end{align*} 
 If we denote $f:=w_t+Lw$, then we have the representation formula:
\be  \label{representation}
w(t,x)=\int_0^t\int_{\mathbb{G}} p(s,y)f(t-s,x\circ y^{-1})dyds.
\ee
From this, we prove the estimate
\be \label{213}
\norm{w}_{L^{q_1}([0,T],L^{p_1}(\mathbb{G}))} \leq CT^{\frac{1}{2}[(2+\frac{2}{q_1}+\frac{Q}{p_1})-(\frac{2}{q}+\frac{Q}{p})]} \norm{f}_{L^q([0,T],L^p(\mathbb{G}))}.
\ee
Let us define a new function $\tilde{p}(t,x)$ defined on $\R \times \mathbb{G}$ via
\begin{align*}
&\begin{cases}
 \tilde{p}(t,x)=p(t,x) \quad &0\leq t\leq T, \\
 \tilde{p}(t,x)=0 \quad &\text{otherwise},
\end{cases}
\end{align*} 
and $\tilde{f}(t,x)$ similarly.
Then, from \eqref{representation}, we have
\be \label{representation'}
|w(t,x)|\leq \int_{\R}\int_{\mathbb{G}} \tilde{p}(s,y)|f|(t-s,x\circ y^{-1})dyds=(\tilde{p}*|f|)(t,x)
\ee
(convolution acts on $\R \times \mathbb{G}$).  Note that for any $1\leq a<\infty$,
\begin{align} \label{s}
\norm{e^{-c\norm{\cdot}^2/t}}_{L^a(\mathbb{G})}= C_0 t^{Q/2a} \norm{e^{-c\norm{\cdot}^2}}_{L^a(\mathbb{G})} = C t^{Q/2a},
\end{align} 
and $\norm{e^{-c\norm{\cdot}^2/t}}_{L^\infty(\mathbb{G})}=1$. Thus, using the heat kernel estimate \eqref{heat}, for any $1\leq a\leq \infty$,
\begin{align} \label{100}
\norm{p(t,\cdot)}_{L^a(\mathbb{G})} &\leq C \frac{1}{t^{Q/2}}\norm{e^{-c\norm{\cdot}^2/t}}_{L^a(\mathbb{G})} =Ct^{-\frac{Q}{2}(1-\frac{1}{a})}.
\end{align}
Let us choose two exponents $1\leq r,s\leq \infty$  such that
\begin{align}\label{rs}
\frac{1}{q_1}+1=\frac{1}{q}+\frac{1}{r},\quad \frac{1}{p_1}+1=\frac{1}{p}+\frac{1}{s}. 
\end{align}
According to the condition \eqref{b condition}, we have $\frac{Qr}{2}(1-\frac{1}{s})<1$.
Therefore, due to \eqref{100},
\begin{align*}
\norm{\tilde{p}}_{L^{r}(\R,L^s(\mathbb{G}))} = \norm{p}_{L^{r}([0,T],L^s(\mathbb{G}))} &\leq C_0 [\int_0^T t^{-\frac{Qr}{2}(1-\frac{1}{s})} dt]^{1/r}=CT^{\frac{1}{r}-\frac{Q}{2}(1-\frac{1}{s})}.
\end{align*}
 Thus, applying the convolution inequality for the mixed-norm spaces  to \eqref{representation'},
 \begin{align*}
\norm{w}&_{L^{q_1}(\R,L^{p_1}(\mathbb{G}))} \leq \norm{\tilde{p}}_{L^{r}(\R,L^{s}(\mathbb{G}))}\norm{\tilde{f}}_{L^{q}(\R,L^{p}(\mathbb{G}))}\nonumber \\
&=CT^{\frac{1}{r}-\frac{Q}{2}(1-\frac{1}{s})}\norm{\tilde{f}}_{L^{q}(\R,L^{p}(\mathbb{G}))}= CT^{\frac{1}{2}[(2+\frac{2}{q_1}+\frac{Q}{p_1})-(\frac{2}{q}+\frac{Q}{p})]}\norm{\tilde{f}}_{L^{q}(\R,L^{p}(\mathbb{G}))}.
\end{align*} 
Thus, we obtain \eqref{213}, which immediately implies \eqref{2131}.

Step 2. The case  $l=k+1$: proof is almost same as the previous case. Applying the heat kernel estimate  \eqref{heat}: for each $1\leq i\leq m$,
\begin{align} \label{heat00}
|Z_ip(t,x)| \leq Ct^{-(1+Q)/2}e^{-c\norm{x}^2 / t}
\end{align}
to the following representation formula
\begin{align} \label{rep00} 
Z_iw(t,x) =\int_0^t\int_{\mathbb{G}} Z_ip(s,y)f(t-s,x\circ y^{-1})dyds,
\end{align}
we can  derive the conclusion as before.
\end{proof}

\section{Heat kernel estimates} \label{section c}
In this Appendix \ref{section c}, we provide  useful estimates related to the semigroup generated by the sub-Laplacian $L$. Let us first derive the $L^p$-estimate on the derivatives of a heat kernel:

\begin{lemma} \label{lemma b.1}
 Suppose that $f$ is a homogeneous function with degree $k$ and $1\leq p<\infty$, $|I|=a\geq 0$. Then, there exists some constant $C$ depending on $f$ such that for any $t>0$,
\begin{align*}
\norm{f Z_I p_t}_{L^p(\mathbb{G})} \leq Ct^{\frac{Q}{2p}+\frac{k-(Q+a)}{2}}.
\end{align*}
Here, for a multi-index $I=(i_1,\cdots,i_a)$ with $1\leq i_1,\cdots,i_a\leq m$, $Z_I$ denotes $Z_{i_1}\cdots Z_{i_a}$.
\end{lemma}
\begin{proof} Recall that under the change of variable $x=D(\sqrt{t})y$, we have $dx=t^{Q/2}dy$. Using this fact and the heat kernel estimate \eqref{heat}, we have 
\begin{align*}
\norm{f Z_I p_t}_{L^p(\mathbb{G})}&\leq C\Big[\int_\mathbb{G} \big(f(x) t^{-\frac{Q+a}{2}} e^{-c\norm{x}^2/t}\big)^p dx\Big]^{1/p} \\
&=Ct^{\frac{k-(Q+a)}{2}} \Big[\int_\mathbb{G} \big(f(y)  e^{-c\norm{y}^2}\big)^p t^{Q/2} dy\Big]^{1/p} =Ct^{\frac{Q}{2p}+\frac{k-(Q+a)}{2}}.
\end{align*}
\end{proof}
Using the previous lemma, we obtain the following lemma, which is a key ingredient in the proof of Proposition \ref{3.5}:
\begin{lemma}\label{lemma b.2}
Suppose that $1<p<\infty$. Then,  for  any $|I|=a\geq 1$, $f\in S^{a-1,p}(\mathbb{G})$,  and $t>0$,
\begin{align}\label{b.2}
\norm{f*Z_Ip_t}_{L^\infty(\mathbb{G})} \leq Ct^{-(\frac{Q}{2p}+\frac{1}{2})}\norm{f}_{S^{a-1,p}(\mathbb{G})}.
\end{align}
\end{lemma}
\begin{proof}
Using the heat kernel estimate \eqref{heat} and the convolution inequality, \eqref{b.2} immediately follows when $a=1$.  Key idea of the proof when $a\geq 2$ is transferring  the directional derivative $Z_I$ from $p_t$ to $f$. We follow the strategy used in the proof of Theorem \ref{3.1}, and use the same notations $Z^R_i$ , $\beta_{ji}$, $Z_{jl}$, $Z_{jI}^R$. 
For each $1\leq i\leq m$ and any smooth functions $g$,
\begin{align} \label{b.3}
f * Z_ig &= f *(\sum_{j=1}^N Z^R_{j} (\beta_{ji}g))=\sum_{j=1}^N \sum_{l,I} Z_{jl}f * (Z^R_{jI} (\beta_{ji} g)).
\end{align}
Let us first prove \eqref{b.2} when $a=2$, and assume that $Z_I=Z_{i_1}Z_{i_2}$, $1\leq i_1,i_2\leq m$. If we denote $p'$ by the conjugate exponent of $p$, then applying the convolution inequality to \eqref{b.3}, 
\begin{align}\label{b.4}
\norm{f*Z_Ip_t}_{L^\infty(\mathbb{G})} =\norm{f*Z_{i_1}(Z_{i_2}p_t)}_{L^\infty(\mathbb{G})}  \leq C\norm{f}_{S^{1,p}(\mathbb{G})}\sum_{j,l,I}  \norm{Z^R_{jI} (\beta_{ji} Z_{i_2}p_t)}_{L^{p'}(\mathbb{G})}.
\end{align}
Note that each $Z_i^R$, $1\leq i\leq m$, can be written as
\begin{align*}
Z_i^Ru = \sum_{j=1}^N Z_j(\gamma_{ji}u)
\end{align*}
for some homogeneous functions $\gamma_{ji}$ of degree $\alpha_j-1$ ($1\leq j\leq N$). Thus, using this fact and applying the product rule to  $Z^R_{jI} (\beta_{ji} Z_{i_2}p_t)$, one can conclude that $Z^R_{jI} (\beta_{ji} Z_{i_2}p_t)$ can be written as the (finite) sum of
$h_k Z_{I_k} p_t$'s
for some homogeneous functions $h_k$ of degree $k-1$ and $|I_k| = k$. Note that according to Lemma  \ref{lemma b.1}, each term $\norm{h_k Z_{I_k} p_t}_{L^{p'}(\mathbb{G})}$ is bounded by
\begin{align*}
 Ct^{\frac{Q}{2p'}+\frac{k-1-(Q+k)}{2}} = Ct^{-(\frac{Q}{2p}+\frac{1}{2})}.
\end{align*}
Thus, combining this with \eqref{b.4}, we obtain \eqref{b.2} in the case $a=2$.

The aforementioned argument works for $a>2$ as well. In fact,  differentiating  \eqref{b.3}  in $Z_i$  directions  ($1\leq i\leq m$)
and then using \eqref{b.5}, we can deduce that
\begin{align*}
\norm{f*Z_Ip_t}_{L^\infty(\mathbb{G})}  \leq C\norm{f}_{S^{l-1,p}(\mathbb{G})} \norm{h}_{L^{p'}(\mathbb{G})},
\end{align*}
where $h$ is the sum of finitely many $h_k Z_{I_k} p_t$'s
for some homogeneous functions $h_k$ of degree $k-1$ and $|I_k| = k$. As mentioned above,  $\norm{h}_{L^{p'}(\mathbb{G})}\leq Ct^{-(\frac{Q}{2p}+\frac{1}{2})}$, which concludes the proof. 
\end{proof}

\section*{Acknowledgement}
The author thanks the advisor Fraydoun Rezakhanlou for sharing interesting ideas. Especially, author thanks  Michael Christ for  helpful discussions regarding Section 3.

\end{document}